\documentclass[twoside,leqno]{article}

\usepackage[letterpaper]{geometry}

\usepackage{siamproceedings}
\usepackage{xcolor}
\usepackage[T1]{fontenc}
\usepackage{amsfonts}
\usepackage{graphicx}
\usepackage{epstopdf}
\usepackage{enumitem}
\usepackage{algorithmic}
\usepackage{mathtools}
\usepackage{amsmath,amssymb}
\ifpdf
  \DeclareGraphicsExtensions{.eps,.pdf,.png,.jpg}
\else
  \DeclareGraphicsExtensions{.eps}
\fi

\newsiamremark{remark}{Remark}
\newsiamremark{hypothesis}{Hypothesis}
\crefname{hypothesis}{Hypothesis}{Hypotheses}
\newsiamthm{claim}{Claim}

\usepackage{amsopn}

\newcommand{\e}{\ensuremath{\epsilon}}

\newcommand{\Vvert}{\vert \hspace{-1pt} \vert \hspace{-1pt} \vert}

\def \1{\mathbf{1}}

\begin{document}

\newcommand\relatedversion{}

\title{\Large Estimation in linear  high dimensional Hawkes processes: a Bayesian approach \relatedversion}
    \author{Judith  Rousseau\thanks{CEREMADE, Universit\'e Paris Dauphine -PSL  and University of Oxford (\email{rousseau@ceremade.dauphine.fr} )} 
    \and Vincent Rivoirard \thanks{CEREMADE, Universit\'e Paris Dauphine -PSL (\email{rivoirard@ceremade.dauphine.fr} )}
    \and D\'eborah Sulem \thanks{Faculty of Informatics, Universita della Svizzera italiana,
  ( \email{deborah.sulem@usi.ch} })
  }

\date{}

\maketitle


\fancyfoot[R]{\scriptsize{Copyright \textcopyright\ 20XX by SIAM\\
Unauthorized reproduction of this article is prohibited}}







\begin{abstract}

    In this paper we study the frequentist properties of Bayesian approaches in linear high dimensional Hawkes processes in a sparse regime where the number of interaction functions acting on each component of the Hawkes process is much smaller than the dimension. We consider two types of loss function: the empirical $L_1$ distance between the intensity functions of the process and the $L_1$ norm on the parameters (background rates and interaction functions). Our results are the first results to control the $L_1$ norm on the parameters under such a framework. They are also the first results to study Bayesian procedures in high dimensional Hawkes processes.
\end{abstract}
\section{Introduction} \label{sec:Intro}

Multivariate Hawkes processes have been widely used both in machine learning and statistics as a versatile model for dependent point processes, with a wide range of applications such as  neurosciences, social sciences and finance.  A $K$ dimensional Hawkes process  $N = (N^1, \cdots, N^K)$  is a marked point process whose intensity process $\lambda_t = (\lambda_t^k)_{k=1}^K$ is defined by 
\begin{equation} \label{intensity}
\lambda_t^{(k)}(f_k)=\nu_k+\sum_{\ell=1}^K\int_{t-A}^{t^-}h_{\ell k}(t-s)dN^{\ell}(s), \quad f_k = (\nu_k, h_{\ell k}, \ell \leq K), \quad t \in \mathbb R
\end{equation}
with $\nu_k \in \mathbb R_+$ is called the background rate on $N^k$  and $h_{\ell k} : \mathbb R_+ \rightarrow \mathbb R_+ $ is the interaction function from $N^\ell$ on $N^k$, for $\ell , k \in [K]= \{1,\cdots, K\}$. 
Recall that $\lambda_t^{(k)} dt = P( N^k \text{ has a jump in } [t, t+dt)| \mathcal G_t) $ where $\mathcal G_t$ is the history of $N$ up to time $t$.

Typically one observes $N$ over a time period $[0,T]$ and the aim is to estimate the true parameter  $f^0 = (f_1^0, \cdots, f_k^0)$, or functionals of $f^0$ such as the graph of interactions or equivalently its adjacency matrix $\Delta(f^0)$ whose components are $ \Delta_{\ell k}(f^0) = \1_{ \|h_{\ell k}^0\|_1 >0}$. Algorithms and statistical methods have been proposed and studied  over the last decade or so but only recently theoretical properties have been derived when the dimension $K$ of the Hawkes process is large. The high dimension of the process poses both theoretical and computational challenges due to the fact that the $K^2$ interaction functions need to be estimated in addition to the $K$ background rates. This question is crucial, for example in neuroscience, where we are able to record the activity of a large number of neurons simultaneously. To address such challenges it is natural to consider a sparse regime where $f$ can be parametrized by a much lower dimensional parameter  and there exist different types of sparsity regimes considered in the literature.  A first type of sparsity is defined by assuming that for all $k\in [K]$ only a small number of interaction functions $h_{\ell k} $ are non null, as in \cite{chen2019}, \cite{cai2024latent}  and \cite{wang2025statistical}, while an alternative type of sparsity consists in assuming that $f$ can be represented by a low dimensional tensor, as in \cite{bacry-etal-2020} and \cite{Tang-Li-2023}. In all but the last paper the results do not allow to directly control the difference between the estimator $\hat f$ and the true parameter $f^0$, when such an estimator is proposed, but either control the empirical $L_2$ distance  $d_{2,T} (\hat f_k, f_k^0) =T^{-1}\int_0^T(\lambda_t^{(k)}(\hat f_k)- \lambda_t^{(k)}( f_k^0) )^2 dt$ or estimate a superset containing  $\{ (\ell, k); \|h_{\ell k}\|_1>0 \}$ or in the case of \cite{wang2025statistical}, under a parametric assumptions on the $h_{\ell,k}$, estimate $h_{\ell,k}$ for a given $(\ell, k)$. As will appear from our results the rates obtained in 
\cite{wang2025statistical} seem suboptimal, but it is to be noted that these rates are not the main objective of \cite{wang2025statistical}.

Therefore the question of the (i) possiblity  to estimate $f$ in high dimensions and (ii) the best achievable rates of convergence in direct norms on $f$ is still open, notably when non parametric assumptions are made on  the interaction functions $h_{\ell k}$. Also, none of the methods cited above are Bayesian. Bayesian nonparametric estimation of multivariate Hawkes processes has been studied theoretically by \cite{donnet18} and \cite{sulem2024bayesian} but these papers deal with  the fixed $K$ case and the dependence of $K $ in their rates is exponential. 

In this paper we study a family of Bayesian approaches to the nonparametric and parametric estimation of high dimensional linear Hawkes processes. We derive posterior contraction rates, which in turns imply the same convergence rates for some Bayesian estimators of $f$. 
We study two types of loss functions, the empirical $L_1$ loss, in a similar spirit to \cite{bacry-etal-2020} and \cite{cai2024latent} (see Theorem \ref{th:gene:d1T}) and the direct $L_1$ loss on $f$ (see Theorem \ref{th:gene:L1}), for which we obtain the oracle rate - up to $\log T$ or $\log K$ terms, i.e. rates that would be obtained if we knew the graph of interactions, i.e. $\Delta (f^0)$. Then we study a number of prior families for which we derive the posterior contraction rates under various smoothness conditions on the true interaction functions $h_{\ell k}^0$.

\section{Bayesian inference for high dimensional Hawkes process}

\subsection{High dimensional Hawkes processes} \label{sec:model}
We observe a  multivariate Hawkes process $N = (N^1, \cdots, N^K)$  where $K$ is possibly large and whose  intensity process $\lambda_t = (\lambda_t^{(k)})_{k=1}^K$ is defined in \eqref{intensity} and depends on $f= (f_1, \cdots, f_K)$. Then the log-likelihood of $N$ at $f$ is given by, see \cite{daley01}, 
\begin{equation}\label{likeli}
L_T(f) = \sum_{k=1}^K L_{T,k}(f_k), \quad L_{T,k}(f_k)= \int_0^T \lambda^{(k)}(f_k)dN_t^k - \int_0^T \lambda^{(k)}(f_k)dt .
\end{equation}
Throughout the paper we assume that the supports of the intensity functions $h_{\ell k}$ are included in $[0,A]$ where $A>0$ is a given constant. 
We denote by $\rho(f) $ the matrix in  $\mathbb R_+^{K\times K}$ defined by $\rho_{\ell k} (f) = \int_0^A h_{\ell k}(x) dx $, is the mass of $h_{\ell k}$. When there is no ambiguity we will write $\rho $ in place of $\rho(f)$. 

Throughout the paper we assume that the true generating process is a Hawkes process as defined in \eqref{intensity} with parameter $f^0 = (f_k^0)_{k\in [K]}$ such that the process is stationary, i.e. the spectral radius  $SpR(\rho^0)<1$ with $\rho^0 = \rho(f^0)$. We write $\mathbb P_0(\cdot)$ (resp. $\mathbb E_0(\cdot ) $) for the probability distribution of $N$ associated to $f^0$  (resp. for any expectation associated to $\mathbb P_0$).




We are interested in the estimation of $f^0$ in the case where $K$ is large but under sparsity assumptions on $\rho^0$ and we focus on Bayesian estimation procedures. Before introducing the family of Bayesian approaches we consider, we present the three main hypotheses on $f^0$ that will be considered throughout the paper. 

\begin{hypothesis}\label{cond:suprho}
We assume that  there exist $1> c>0$ and $R_\infty, R_1<\infty$  such that $\rho^0= \rho(f^0)$  verifies: for all $n\geq 1$ 
$$ \Vvert (\rho^0)^n\Vvert_{\infty } = \max_{\ell \leq K} \sum_{k=1}^K ((\rho^0)^n)_{\ell, k} \leq R_\infty c^n, \quad  \Vvert (\rho^0)^n\Vvert_{1 } =  \Vvert (\rho^{0Tr})^n\Vvert_{\infty }  \leq R_1 c^n $$
where $\rho^{0Tr} = \rho(f_0)^{Tr}$ is the transpose of $\rho^0$ and $\Vvert M\Vvert_{\infty} $ (resp. $\Vvert M\Vvert_{1} $) is the operator norm of the matrix $M$ related to the norm $\| \cdot \|_\infty $ (resp. $\| \cdot \|_1$). 
\end{hypothesis}

\begin{hypothesis}\label{cond:sparsity1}
 We denote $ S_0(k) = \{ \ell, \rho_{\ell k}^0 > 0 \} \subset [K]$ and $S_0$ the list made of the sets $S_0(k)$, $k\in [K]$.
 
There exists $s_0< \infty$  such that for all $k$, 
$ |S_0(k)| \leq s_0 $, where $|S_0(k)|$  is the cardinal of $S_0(k)$.

\end{hypothesis}

\begin{hypothesis}\label{cond:sparsity2}
There exist $\bar h_0, c_0,c_1 <\infty $ and $A>0$  such that $\sup_{x>A}h_{\ell k }^0(x) = 0$,
$$\max_{\ell, k} \|h_{\ell, k}^0\|_\infty \leq \bar h_0 < \infty , \quad  c_0 \leq \inf_k\nu_k^0 \leq \sup_k\nu_k^0 \leq c_1  .$$
These constants do not depend on $K$ and $T$.
\end{hypothesis}

A consequence of Hypotheses \ref{cond:suprho} and  \ref{cond:sparsity2} is that we can bound uniformly $\mu_k^0 = \mathbb E_0(\lambda_t^{(k)}f_k^0))$ and  $\mathbb E_0(\lambda_t^{(k)}f_k^0)^2)$ and some exponential moments of $N^k( [0,B) )$ for $B>0$.
\begin{lemma}\label{lem:lambda2}
Assume that  Hypotheses \ref{cond:suprho}-\ref{cond:sparsity2} hold , then 
\begin{align*}
\mu_k^0&=E_0[\lambda_0^{(k)}(f_k^0)] \leq \frac{ c_1 ( 1 -c + R_1)}{ 1-c}  =: C_0 < \infty ; \quad 
\mathbb E_0[\lambda_0^{(k)}(f_k^0)^2] \leq C_0^2+\frac{  C_0 \bar h_0 R_\infty R_1 c }{ 1 -c} =: \bar C_0\\
 \mathbb E_0[e^{t N^k ( 0, B)}] & \leq e^{ t\gamma B } ,\quad  \forall 0\leq  t \leq t(c):= \frac{ \log (1+ c) - \log (2c)}{1 + \frac{ 2R_\infty }{1 - c}}, \quad \gamma  = \frac{ 2 C_0 R_1  }{ 1-c}
\end{align*}
\end{lemma}

\begin{proof}[Proof of Lemma \ref{lem:lambda2}]
First write   $\mu^0 = (I - \rho^{0Tr})^{-1} \nu^0$, so that 
$
\mu_i^0 \leq \|\mu^0\|_\infty \leq \Vvert(I - \rho^{0Tr})^{-1}\Vvert_\infty \|\nu^0\|_\infty \leq c_1 \sum_{n=0}^\infty \Vvert (\rho^{0Tr})^n\Vvert_\infty \leq c_1 (1 + R_1 /(1-c))
$ which implies the first part of Lemma \ref{lem:lambda2}. We now bound the second moment of $\lambda_t^{(k)}(f_k^0)$. Since, 
$\lambda_0^{(k)}(f_k^0)^2 = [ \nu_k^0 + \sum_{\ell \in S_0(k)} \int_{-A}^{0^-}h_{\ell k}^0(-s)dN_s^\ell ]^2 $, 
and
$\nu_k^0+\sum_{\ell \in S_0(k)} \rho_{\ell k}^0\mu_\ell^0= \mu_k^0,$
\begin{align*}
\mathbb E_0[\lambda_0^{(k)}(f_k^0)^2 ] 
& =  (\mu_k^0)^2  + \sum_{\ell_1, \ell_2 \in S_0(k)}\int_{[-A,0)^2}h_{\ell_1 k}^0(-s_1)h_{\ell_2 k}^0(-s_2)\nu^{\ell_1, \ell_2}(s_1-s_2) ds_1ds_2\\
\end{align*}
where $\nu$ is defined and studied in \cite{bacry-etal-2020} 
 $$ \nu^{\ell_1, \ell_2}(s_1-s_2)ds_1ds_2 = \mathbb E_0[dN_{s_1}^{\ell_1}dN_{s_2} - \mu_{\ell_1}^0\mu_{\ell_2}^0ds_1ds_2 ] - \mu_{\ell_1}^0\1_{\ell_1=\ell_2}\delta_0(s_1-s_2)ds_1$$
 and verifies $\nu  = diag(\mu^0) g$ with $diag(\mu^0)$ the diagonal matrix defined by the vector $\mu^0$ and with
 $$  g^{\ell_1 \ell_2}(t)  = h_{\ell_1\ell_2}^0(t)+ \sum_{j}\int_{0}^{A}h_{\ell_1j}^0(s)g^{j \ell_2}(t-s)ds . $$
We define the matrix $g_\infty  = (\|g^{\ell_1 \ell_2}\|_\infty )_{\ell_1, \ell_2}$ (and similarly for $h^0_\infty$) and we  bound  (componentwise)
 $g_\infty \leq h^0_\infty  + \rho^0 g_\infty$ which in turns implies that componentwise
 $$g_\infty \leq \sum_{n=0}^n (\rho^0)^n  h^0_\infty \quad \text{and } \quad \Vvert g_\infty\Vvert_\infty \leq \frac{  \bar h_0 R_\infty }{ 1 -c}$$
 using Assumption \ref{cond:suprho}. Then 
\begin{align*}
\mathbb E_0[\lambda_0^k(f_0)^2 ] & \leq (\mu_k^0)^2  +\frac{  \bar h_0 R_\infty }{ 1 -c} \sum_{\ell_1, \ell_2} \mu_{\ell_1^0} \rho^0_{\ell_1 k}\rho^0_{\ell_2 k} \\
& \leq   C_0^2  +\frac{  \bar h_0 R_\infty C_0 }{ 1 -c} (\sum_{\ell}  \rho^0_{\ell k}) = C_0^2  +\frac{  \bar h_0 R_\infty C_0  \Vvert \rho^{0Tr} \Vvert  }{ 1 -c} 
C_0^2+\frac{  C_0 \bar h_0 R_\infty R_1 c }{ 1 -c}:= \bar C_0
\end{align*}
 To bound $\mathbb E_0(e^{t N^k [0,B)})$ we use inequality (3.16) of Theorem 3.7 of \cite{leblanc2025} which implies that for all $t < t(c)$
 \begin{align*}
     \mathbb E_0 \left[e^{t N^k [0,B)}\right] & \leq e^{ BC_0 \| \left( Id - \frac{1+c}{2c} \rho^0\right)^{-1}t e_k\|_1 }
 \end{align*}
 where $e_k\in \mathbb R^K$ is equal to $e_k(j)=\1_{j =k}$, $j \in [K]$. Bounding 
 $$\| \left( Id - \frac{1+c}{2c} \rho^0\right)^{-1}t e_k\|_1\leq t\sum_{n=0}^\infty \left( \frac{1+c}{2c}\right)^n \Vvert (\rho^0)^n \Vvert_1  \leq \frac{2 t R_1}{1-c} $$
 terminates the proof.
  \end{proof}

Assumption  \ref{cond:suprho} is an important assumption also to control the global behaviour of $N[a,b]$ for fixed intervals $[a,b]$ and is used in \cite{leblanc2025}.
If  the constants $R_1, R_\infty$ were allowed to depend on $K$, then this assumption would be equivalent to  $SpR(\rho^0)< 1$ with $c$ only slightly larger than $SpR(\rho^0)$. Here the additional assumption is on the dimension free constants $R_1, R_\infty$.

 In the  papers on high dimensional Hawkes processes (\cite{chen2019,bacry-etal-2020,cai2024latent,Tang-Li-2023}) the authors consider the stronger conditions:
 $ \Vvert (\rho^0) \Vvert_{\infty } <1$ and often $ \Vvert (\rho^0)\Vvert_{1 } <1 $ which imply Hypothesis  \ref{cond:suprho} with $R_1 = R_\infty =1$. Hypothesis \ref{cond:sparsity1} is the common sparsity condition stating that each dimension is directly impacted by  at most $s_0$ components.

\subsection{Bayesian inference} \label{sec:Bayes}
The Bayesian inference consists in defining a probability distribution $\Pi$, called the prior, on the set of parameters $\mathcal F^K$ with 
$$\mathcal F = \mathbb R_{+*} \times \mathcal H^K; \quad \mathcal H = \{ h : [0,A] \rightarrow \mathbb R_+, \int_0^A h(x) dx < 1  \},$$
where $\mathbb R_{+*} = ]0,\infty[$.

Both for practical and theoretical reasons we restrict to a prior in the form 
 \begin{equation}\label{prior:structure0}
     \Pi(df) = \prod_{k=1}^K \Pi_k(df_k),
 \end{equation} 
 so that the $f_k$'s are a priori independent and where $\Pi_k$ has support included in $\mathcal F$. 

 Then for all $B = B_1 \otimes \cdots \otimes B_K$ with $B_k \subset \mathcal F$,  the posterior distribution defined as the conditional distribution of $f$ given the observation $N$ has the form, see \cite{donnet18},
$$ \Pi ( B | N) = \prod_{k=1}^K\Pi_k ( B_k | N), \quad \Pi_k ( B_k | N)  = \frac{ \int_{B_k} e^{L_T^k(f^k)} d\pi_k(f_{k} ) }{ \int_{\mathcal F} e^{L_T^k(f^k)} d\pi_k(f_{k}) }. $$
Throughout the paper we denote the posterior expectation of a functional $G(f)$  by $E^{\pi}( G(f)|N)$. 
The product structure of the posterior distribution is crucial from a practical point of view since it allows for parallel computing where each posterior $\Pi_k ( \cdot | N) $ can be approximated independently. From a theoretical point of view it is also important since it allows to study each posterior separately and thus to be avoid the impact of $K^2$ and replace it with $K$. In particular we do not restrict the posterior to the set of stationary Hawkes processes which is 
 $ \mathcal F^K \cap \{ Sp(\rho(f)) < 1 \} $.

 For the sake of simplicity we assume that for each $k$ the prior $\Pi_k$ can be written as 
\begin{equation}\label{prior:structure}
 \Pi_k(d\nu_k, dh_{\ell k}, \ell \leq K) = \pi_\nu(\nu_k)d\nu_k \times \Pi_{k,h}( dh_{\ell k}, \ell \leq K).
 \end{equation}
In other words $\nu_k $ and $(h_{\ell k},\ell \leq K)$ are a priori independent and the prior on $\nu_k$ has a density with respect to Lebesgue measure. This hypothesis can easily be changed. 

To induce sparsity in the prior and because the graph of interactions is unknown a priori, we consider the following family of prior distributions $\Pi_{k,h}$ which we call the model selection priors. 
\begin{definition} [Model selection prior]
The prior $\Pi_{h,k}$ is a model selection prior if it has the following hierarchical structure
 \begin{align} \label{model:prior}
     S &\sim \Pi_S \quad \text{ over} \quad S\subset [K] \nonumber\\
     \forall \ell \in S, \quad h_{\ell k} &\stackrel{iid}{\sim} \Pi_h, \quad \text{and } \quad \forall \ell \in S^c ,\quad h_{\ell k}=0.
 \end{align} 
\end{definition}

The second part of the definition, i.e. $\forall \ell \in S, \quad h_{\ell k} \stackrel{iid}{\sim} \Pi_h $, is not crucial and corresponds in fact to the most complex prior, given the sparsity pattern $S$; since it does not make any structural form on $h_{\ell k}$ when $\ell \in S$. For instance  instead one could consider the following variant: $h_{\ell k} = \beta_{\ell k} h$, with $\beta_{\ell k}\in \mathbb R_+$  and $h$ is a function on $[0,A]$,
 \begin{align}\label{variant1}
     S &\sim \Pi_S \quad \text{ over} \quad S\subset [K], \quad h \sim \Pi_h \nonumber\\
     \forall \ell \in S, \quad \beta_{\ell k}  &\stackrel{iid}{\sim} \Pi_\beta, \quad \text{and } \quad \forall \ell \in S^c ,\quad \beta_{\ell k}=0.
 \end{align} 
 This prior imposes more structure on the parameter $f_k$ and is therefore more restrictive; but it may lead to faster rates of convergence. More discussion is provided in Section \ref{sec:priors}. 

A first step to understand the (asymptotic) behaviour of the posterior distribution is to derive posterior contraction rates for the parameters of interest, which are defined as positive sequences $u_T$ such that for some loss function $d(f,f')\geq 0$ 
\begin{equation}\label{post:contraction}
\mathbb E_0[\Pi( d(f, f_0)\leq u_T| N) ] = 1 + o(1).
\end{equation}
Then as in Theorem 2.5 of \cite{ghosal00}, one can typically  construct an estimator $\hat f$ based on the posterior distribution such that $d(\hat f, f_0)$ converges to 0 at the rate $u_T$ in probability. 
Moreover deriving posterior contraction rates as in \eqref{post:contraction} is a necessary first step for a finer analysis of the posterior uncertainty quantification based on credible regions, see for instance \cite{Hoffmann2013OnAP,Rousseau2016AsymptoticF}.

\subsection{General result on posterior contraction rates in $d_{1,T}$.}

In this section we use the theory developed in \cite{ghosal:vdv:07} and extend the results obtained in \cite{donnet18} for multivariate Hawkes processes with finite $K$ to the high dimensional setup on posterior contraction rates in  the empirical $L_1$ norm: 
 \begin{equation}\label{d1T}
  d_{1,T} (f_k, f_k^0) = \frac{ 1}{T} \int_0^T |\lambda_t^{(k)}(f_k) -\lambda_t^{(k)}(f_k^0)|dt .\end{equation}
  
\begin{theorem}\label{th:gene:d1T}
Consider  $N =(N^k, k\in [K])$ a stationary multivariate Hawkes process with true parameter $f^0$ which satisfies Hypotheses \ref{cond:suprho}, \ref{cond:sparsity1} and \ref{cond:sparsity2}. 

Consider a Prior $\Pi(df) =\prod_{k=1}^K \Pi_k (df_k) $ and such that for each $k$, $\Pi_k $  verifies \eqref{prior:structure} where the prior density  $\pi_\nu$ is positive and continuous on $\mathbb R_{+*}$ and   the prior $\Pi_{k,h}$ on $h_k = (h_{\ell k}, \ell \in [k])$ is a model selection prior and  verifies: There exists a sequence  $\e_T=o(1)$ and such that $T\e_T^2 \gtrsim (\log T)^3$ 
\begin{itemize}
\item (i) there exist   constants $\kappa, M_\infty, C_s>0$ such that 
 \begin{align*}
     \Pi_S( S=S_0(k) ) &\gtrsim e^{- C_s s_0\log K} \\
 \min_{\ell \in S_0(k)} \Pi_h\left( \| h_{\ell k } - h_{\ell k}^0\|_2 \leq \e_T ; \|h_{\ell k}\|_\infty \leq M_\infty  \right) &\gtrsim e^{ - \kappa Ts_0 \e_T^2 \log(\log T + s_0)- C_s s_0 \log K}
 \end{align*}
\item (ii) there exist $\mathcal H_T \subset \mathcal H$, $0<\delta\leq 1/C_0$  and $0<L_T \leq \delta \log T $ such that $0\in \mathcal H_T$ and for some $\zeta >0$ and $M_T $ going to infinity
 $$ \mathcal N( \zeta \e_T/\sqrt{L_T}, \mathcal H_T, \| \cdot \|_1) \leq T \e_T^2, \quad \Pi( \mathcal H_T^c) \leq  e^{ - M_T s_0^2 T \e_T^2 \log(\log T+ s_0)- M_T s_0 \log K } $$
where $\mathcal N(u,\mathcal H_T, \| \cdot \|_1) $ denotes the log-covering number of $\mathcal H_T$ by $L_1$ balls of radius $u$. 
 \item (iii)  For all $k$, there exists $M_T$ going to infinity such that 
 $$\Pi_S( |S| > L_T ) \lesssim  e^{ - M_T s_0^2 T \e_T^2 \log(\log T + s_0)- M_T s_0 \log K} $$
\end{itemize}

Then if $\log K = o(\sqrt{T}) $, $L_T>s_0$ verify 
\begin{equation}\label{eq:postrate:d1T}
v_T = \sqrt{ \log T} M_T' \left( (s_0 + \sqrt{L_T})\e_T+ \sqrt{ \frac{ s_0\log (K)}{T} }\right) = o(1)
\end{equation}
then 
$$E_0\left( \Pi_k( d_{1,T}( f_k, f_k^0) > v_T |N )  \right) = o(1) .$$
\end{theorem}

\begin{proof}[Proof of Theorem \ref{th:gene:d1T}]
The proof of Theorem \ref{th:gene:d1T} is an adaptation of that of Theorem 1 in \cite{donnet18} to control the dependence on $K$ which is exponential in  \cite{donnet18}. 
if $f_k \in B_2(f_k^0; \e_T) \cap \{ |\nu_k - \nu_k^0|\leq \e_T^2\}$ where 
$ B_2(f_k^0; \epsilon) = \{h_k; \max_{\ell \in S_0(k)}\| h_{\ell k } - h_{\ell k}^0\|_2 \leq \epsilon ;max_{\ell \in S_0(k)}\| h_{\ell k } \|_\infty \leq M_\infty;  \,\sum_{\ell \notin S_0(k)}\rho_{\ell k } \leq \epsilon^2 \} ,\quad \epsilon>0 $
then we bound with  $\Psi(u):=-\log(u)-1+u\geq 0$, for $u>0$ 
\begin{align}\label{KL1}
KL_T(f_k^0,f_k ) &=  \mathbb E_0[L_{T, \ell}^{(k)}(f_k^0)-L_{T, \ell}^{(k)} (f_k)]  = T\mathbb E_0\left[\Psi\left(\frac{\lambda_0^{(k)}(f_k)}{\lambda_0^{(k)}(f_k^0)}\right)\lambda_0^{(k)}(f_k^0)\right]\nonumber \\
& \leq  T\mathbb E_0\left[\Psi\left(\frac{\lambda_{0}^{(k)}(f_{A_0}) }{\lambda_0^{(k)}(f_k^0)}\right)\lambda_0^{(k)}(f_k^0) \right]+ T \sum_{\ell\notin S_0(k)^c} \mu_\ell^0\rho_{\ell, k}
\end{align}
where we have used the fact that $\Psi(u+v)\leq \Psi(u)+v$ when $u,v\geq0$ and 
$\lambda_{0}^k(f_{A_0})  = \nu_k + \sum_{\ell\in S_0(k)} \int_{-A}^{0-} h_{\ell, k}(-u)dN^\ell_u.$
We bound the first term of the right handside of \eqref{KL1} similarly  to Lemma 2 of \cite{donnet18} and the second term is bounded by definition of $ B_2(f_k^0; \e_T)$. To use the computations of \cite{donnet18}, we need only to control $\max_{\ell \in S_0(k)} \sup_{[0,T]} N^\ell([t-A, t))$. This is done by bounding, using Lemma \ref{lem:lambda2},
 \begin{align*} 
 \mathbb P_0 \left( \max_{\ell \in S_0(k)} \sup_{[0,T]} N^\ell([t-A, t)) \leq C \log T\right) &\leq s_0 T \max_{\ell} \mathbb P_0 \left(  N^\ell([0,2A)) > C \log T \right) \\
 & \leq  s_0 Te^{-  t(c) C\log T/4} \leq T^{-\alpha}
 \end{align*}
 as soon as $t(c)C >8(\alpha +1) $.
 We then obtain that 
 \begin{equation}\label{KL2}
     KL_T(f_k^0,f_k )\leq T\e_T^2  \kappa_0s_0^2 \log (\log T+ s_0)  
 \end{equation}
 where $\kappa_0$ is a constant depending on $c_1, c_0, \bar C_0, \bar h_0$. Note that we have in fact used a weaker condition on the prior than the model selection prior, since instead of restricting to $\sum_{\ell \in S_0(k)^c}\rho_{\ell k}=0$ (i.e. $S= S_0(k)$) we only need lower bound
 $\Pi(\sum_{\ell \in S_0(k)^c}\rho_{\ell k}\leq \e_T^2) $. 

Let $\delta_T= o(1)$ and define
\begin{equation}\label{OmegaT}
\Omega_{T}  =  \cap_{k\in [K]} \max \left( \left|\frac{N^k(-A, T] }{ T}  - \mu_k^0\right|,  \left|\frac{N^k[0, T] }{ T}  - \mu_k^0\right|,  \left|\frac{N^k(0, T-A] }{ T}  - \mu_k^0\right| \leq \delta_T\right)
\end{equation}
 Then  by Lemma \ref{lem:devNT}, $\mathbb P_0 (\Omega_T^c) = o(1)$.
  Then as in the proof of Theorem 1 of \cite{donnet18}, choosing $z_T =o(M_T)$ going to infinity , for any $B_k \subset \mathcal F^k$v which may depend on $N$,
\begin{equation} \label{gene:bound}
\begin{split}
\mathbb E_0( \Pi_k ( B_k | N)\1_{\Omega_{T}}) &  \leq \mathbb E_0(\phi_T) + \frac{ 1 }{z_T}  \\ 
 & \quad + e^{2z_T s_0^2 T\e_T^2\log (\log T+ s_0) +  z_Ts_0 \log K} \left[\int_{\mathcal F_T} \mathbb E_{\tilde f_{[k]}}(\1_{B_k}(1 -\phi_T)\1_{\Omega_{T}} )d\pi_k(f_{k} ) + \Pi( \mathcal T_F^c)\right].
\end{split}
\end{equation}
 and where $\tilde f_{[k]} = (f_{1}^0,\cdots,   f_{k}, \cdots , f_{K}^0)$. In other words the components $\tilde f_{k'} = f_{k'}^0$ for $k' \neq k$ and $\tilde f_{k} = f_{k}$ and $\phi_T $ are test functions in $[0,1]$. 
 Let $\mathcal F_T = \cup_{S\subset [K], |S|\leq L_T} \mathcal F_T(S) \subset \mathcal F,$ with
\begin{equation*}
\mathcal F_T(S) = \{f_k;   \forall \ell \in S,  \, h_{\ell k} \in \mathcal H_T, \, \forall \ell \notin  S, h_{\ell k}=0\}, 
\end{equation*}
 then under assumptions (ii) and  (iii) of Theorem \ref{th:gene:d1T} 
 \begin{align}\label{piFTc}
 \Pi( \mathcal F_T^c) 
 & \leq  \Pi_S( |S|> L_T)  + \Pi( \mathcal H_T^c)L_T \nonumber \\
 &\leq   e^{ -M_T s_0^2T \epsilon_T^2 \log T - } ( 1 + L_T) \leq  e^{ -M_Ts_0^2 T \epsilon_T^2 \log( \log T+s_0)/2} .
 \end{align}
The tests $\phi_T$ are constructed similarly to \cite{donnet18} but the control of the errors analysis is refined in terms of the dependence on $K$. Let $j \geq 1$,  on $S_{j}(k) = \{ f_k; d_{1,T}(f_k, f_k^0) \in (j \e_T, (j+1)\e_T)\}$, then $ \int_0^T \lambda_t^{(k)}(f_k) dt \leq  \int_0^T \lambda_t^{(k)}(f_k^0) dt + T(j+1)\e_T$  so that on $\Omega_T$, defined in \eqref{OmegaT},
   $$ \nu_k + \sum_{\ell \in [K]} \rho_{\ell,k} \frac{N^\ell[0, T-A]}{T} \leq \mu_k^0 + (j+1)\e_T + \delta_T/c_0 $$
and if $f_k\in S_j(k)$,  $f_k \in \mathcal F_j(k)$ with 
$$ \mathcal F_j(k) = \{ f_k \in \mathcal F_T; \nu_k \leq \mu_k^0 + \delta_T/c_0+(j+1)\e_T\}.$$
Consider a covering of $\mathcal F_j(k)$, $(f_{i,k})_{i\leq \mathcal N_j}$ with $\| \cdot \|_1$ balls of radius $\zeta j \epsilon_T$  and $\phi_j = \max_{i\leq \mathcal N_j} \phi_{f_{i,k}}$, with $\phi_{f_{i,k}}$ defined in Lemma 1 of \cite{donnet18} then there exists $x_1 $ depending in $C_0, c_0$  such that 
$$ \mathbb E_{0} ( \phi_j \1_{\Omega_{1,T}} ) \leq \mathcal N_j e^{- x_1 j T\e_T \min(j\e_T,1)  } $$
and for all $f_k \in   \mathcal F_j(k)$, 
\begin{equation} \label{type2error}
\mathbb E_{\tilde f} ( (1-\phi_j) \1_{\Omega_{T}}\1_{f_k\in S_j} )  \leq e^{- x_1 jT\e_T \min (j \e_T, 1) } \end{equation}
We now bound $\mathcal N_j$. Let $S\subset [K]$ such that $|S|\leq  L_T$, 
$\mathcal F_j(S) = \{ f_k \in \mathcal F_j \cap \mathcal F_T(S) \} $, by definition $\mathcal F_j = \cup_{|S|\leq L_T} \mathcal F_j(S)$. Consider $\zeta >0$  and  $f_k^1 , f_k^2 \in\mathcal F_j(S)$  such that $  \max_{\ell \in S}\|h_{\ell k}^1 -h_{\ell k}^2\|_1\leq (j\zeta \e_T)/(2|S|)$ then 
\begin{align*}
\|f_k^1 -  f_k^2 \|_1 
& \leq  \sum_{\ell \in S}\|h_{\ell,k}^1 -h_{\ell k}^2\|_1  \leq  \zeta \e_Tj 
\end{align*} 
 so that choosing $j \geq \sqrt{L_T}$, 
$\mathcal N_j(S)  \leq  \mathcal N( \zeta\e_T/\sqrt{L_T}, \mathcal H_T, \|\cdot \|_1)^{|S|}\leq e^{T \e_T^2 L_T }$ and 
\begin{align*}
\mathcal N_j&\leq \sum_{s=0}^{L_T}\frac{ K! }{ s! (K-s)!}  e^{T \e_T^2 s} \lesssim e^{T \e_T^2 s_1}K^{s_1}+ \sum_{s=s_1}^{L_T}\frac{ 1 }{ \sqrt{s} } e^{  s\log( K/s) } e^{T \e_T^2s }\\
& \lesssim L_Te^{ L_T [T\e_T^2 +\log (K/L_T)]}
\end{align*}
Therefore, since $\log K = o(\sqrt{T})$   for all $j>J_0\sqrt{ L_T}max( 1, \sqrt{\log (K/L_T)/(T \e_T^2)})  $ with $J_0$ large enough
$\mathbb E_0(\phi_j) \lesssim  e^{ - x_1 j T\e_T ( j\e_T, 1)/2}$
which, combined with \eqref{gene:bound} and \eqref{type2error}  terminates the proof of Theorem \ref{th:gene:d1T}.
\end{proof}

\begin{remark}
    Condition (i) in Theorem \ref{th:gene:d1T} has two components: the term $O(s_0 \log K)$ which reflects the difficulty to select the set of active indexes $S_0(k)$ under the prior, which can call the combinatorial complexity and the term $ T\e_T^2 \log(\log T + s_0)$ which reflects typically the capacity of the prior on each $h_{\ell k}$ to approximate $h_{\ell k}^0$. When $\log K$ is larger than $T\e_T^2 $ then the combinatorial complexity dominates the rate, otherwise it is the \textit{nonparametric} aspect of $h_{\ell k}$ which dominates the rate, at least as far as the Kullback-Leibler condition  is concerned. 
\end{remark}
\begin{remark}\label{Uniforprior}
    A very common model selection prior consists in defining $\Pi_S$ by first selecting $|S| \sim \Pi_{1,S}$ and then given $|S| = s$ considering a uniform prior on $S$. The latter has probability mass function 
 $$\Pi_S( S | |S|=s )  =\frac{ s! (K-s)! }{ K!} \asymp \frac{(K/s)^s}{ \sqrt{s} }, \quad 1 < s = o(K).$$
Therefore such a family of priors verify the first part of assumption (i) as soon as $\Pi_1( s  = s_0 ) \geq e^{-C s_0 \log K}$ for some $C>0$, which is a very weak assumption.
\end{remark}
The rate $s_0 \e_T$ (up to $\log T$ terms) can be interpreted as the posterior contraction rate which would be obtained if the the active set of interactions $S_0$ were known. Indeed, in this case  the priors $ \Pi_k  $ on $f_k$ are such that   $(h_{\ell k},\ell \notin S_0(k)) = 0$ and only a prior on $(h_{\ell k},\ell \in S_0(k)) $ needs to be defined.  


\begin{proposition}\label{oracle}
Consider the  Hypotheses \ref{cond:suprho}, \ref{cond:sparsity1} and \ref{cond:sparsity2} and that $S_0$ is known. Assume  that the prior density  $\pi_\nu$ is positive and continuous on $\mathbb R_{+*}$  and  that the prior on $h_k = (h_{\ell k}, \ell\in  S_0(k))$ verifies: There exists a sequence  $\e_T=o(1)$ and such that $Ts_0^2\e_T^2 \gtrsim (\log T)^3$ 
\begin{itemize}
\item (i) there exist constants $\kappa, M_\infty>0$ such that 
 $$\min_{\ell\in S_0(k)}\Pi\left( \| h_{\ell k } - h_{\ell k}^0\|_2 \leq \e_T ; \,  \| h_{\ell k}\|_\infty \leq M_\infty\right) \gtrsim e^{ - \kappa T s_0\e_T^2 \log T},$$
\item (ii) there exist $\mathcal H_T \subset \mathcal H$ and $M_T$ going to infinity such that for some $\zeta >0$
 $$ \mathcal N( \zeta \e_T/s_0, \mathcal H_T, \| \cdot \|_1) \leq T \e_T^2, \quad \Pi( \mathcal H_T^c) \leq  e^{ - M_T Ts_0^2 \e_T^2 \log T}$$
\end{itemize}
Then if $\log K = o( \sqrt{T})$ and $s_0 \e_T \log( \log T+ s_0)= o(1)$ , for all $M_T'$ going to infinity,  
\begin{equation*}
E_0\left( \Pi_k( d_{1,T}(f_k,f_k^0) \geq  M_T' s_0 \e_T (\log T)^{1/2} |N ) \right) = o(1) .
\end{equation*}
\end{proposition}

Hence, up to $\log T$ terms the rate $s_0\e_T$ can be seen as the rate that would be obtained if the connectivity graph was known;  and $\e_T$ is determined by the ability of the prior on $h_{\ell k} $ to approximate $h_{\ell k}^0$, see Section \ref{sec:priors} for illustrations of this. In parametric models, one typically has for all $H$
$$ \Pi\left( \max_{\ell \in S_0(k)}\| h_{\ell k } - h_{\ell k}^0\|_2 \leq \ T^{-H} ; \, \max_{\ell\in S_0(k)} \| h_{\ell k}\|_\infty \leq M_\infty\right) \gtrsim e^{-s_0 \kappa_1\log T}$$
for some $\kappa_1>0$ so that, $\e_T =O( \log T/ \sqrt{T} )$ but more importantly the $d_{1,T} $ posterior contraction rate is $O\left( \frac{\sqrt{s_0 }\log T}{\sqrt{T}}\right) $ instead of $O\left( \frac{ s_0 \log T}{\sqrt{T}}\right) $  and $s_0$ is allowed to grow like $o( T/(\log T)^2)$   when $S_0$ is known and similarly, when $S_0$ is not known, the rate is $O( \sqrt{L_T \log K/T} )$ instead of, 
 $(s_0 + \sqrt{L_T \log K} )/\sqrt{T}$. 


In the following section we derive a posterior contraction rate in the direct $L_1$ norm.

\subsection{ General result on posterior concentration rates under the $L_1$ loss.}\label{sec:mainL1:th}
In this section we study the \textit{direct} loss function 
 $$ \|f_k - f_k^0 \|_1 = |\nu_k - \nu_k^0| + \sum_{\ell \in [K] } \|h_{\ell k} - h_{\ell k}^0\|_1.$$

\begin{theorem}\label{th:gene:L1}
Consider  $N =(N^k, k\in [K])$ a stationary multivariate Hawkes process with parameter $f^0$ which satisfies Hypotheses \ref{cond:suprho}, \ref{cond:sparsity1} and \ref{cond:sparsity2}. 

Consider a prior whose structure follows \eqref{prior:structure} and such that  $\pi_\nu$ is positibe and continuous and a model selection prior on $h_k$. Let $k\in [K]$, assume  that the prior on $h_k = (h_{\ell k}, \ell \in [K])$ verifies (i)-(iii) of \cref{th:gene:d1T}. 
Then if $\log K = (\sqrt{T})$ and  $\log K = O ( \min( e^{-6AC_0L_T} \e_T^{-1}/\log T, T^{1/3}e^{-4AC_0L_T}) $, if   $s_0\leq L_T \leq \delta \log T$ with $\delta < 1/(6AC_0)$ and 
$e^{6 AC_0 L_T} v_T (\log T+ \log K) = o(1)$, then for some  $M>0$ going to infinity
\begin{equation}\label{eq:postrate:L1}
E_0\left( \Pi_k\left( \|f_k - f_k^0 \|_1> M e^{6 AC_0 L_T} v_T (\log T+ \log K) |N \right) \right) = o(1),
\end{equation}
where $v_T$ is defined in \eqref{eq:postrate:d1T}. 

In particular if $s_0 \leq L_T \leq \delta \log \log T$, then there exists $q\geq 0$
\begin{equation}\label{eq:postrate:L1:2}
E_0\left( \Pi_k\left( \|f_k - f_k^0 \|_1>  M (\log T)^q(1 + \log K) \left(\e_T + \sqrt{\frac{ \log K}{ T}} \right)|N \right) \right) = o(1),
\end{equation}

\end{theorem}

\begin{remark}\label{K-LT}
The constraint $\log K =O( Te^{-6AC_0L_T} /\log T) $ is very weak, since $L_T$ is prior dependent and need only be larger than the sparsity  $s_0$. However an important consequence of \cref{th:gene:L1} is that $s_0 $ impacts the $L_1$ rate exponentially. It is unclear if the dependence on $s_0$ in the rate is sharp or if it is an artifact of the proof. Under the assumption that $s_0$ is bounded, then one can choose  a prior so that $L_T\leq \delta \log \log T $ which implies a posterior contraction rate in $L_1$ of order $O( (\e_T \vee \sqrt{(\log K)/T} )(\log T)^q)$  as long as $\log K  = o( \sqrt{T})$ for some $q >0$ which depend on $C_0$ and $\delta$.
\end{remark}

\begin{proof}[Proof of Theorem \ref{th:gene:L1}]

For all $S \subset [K]$ and $B \subset \mathbb R$, denote hereafter $N^S(B) =\sum_{\ell \in S} N^\ell(B)$. 

 First from Theorem \ref{th:gene:d1T}, on $\Omega_T$, for all $k$, 
$$ \Pi( d_{1,T}(f_k;f_k^0)\leq v_T|N ) = 1 + o_{P_0}(1). $$
Recall that
\begin{align*}
 d_{1,T}(f_k;f_k^0) &= \frac{1}{T} \int_0^T \left|\nu_k - \nu_k^0 + \sum_{\ell \in S\cap S_0(k)^c} \int^{t^-} h_{\ell k}(t-s)dN^\ell_s + \sum_{\ell\in S_0(k)}(h_{\ell k}(t-s)- h_{\ell k}^0(t-s))dN^\ell_s\right|dt
 \end{align*}
 To derive a concentration rate in $L_1$, we first derive a concentration rate on  $|\nu_k - \nu_k^0|$, then on $\sum_{\ell \in S_0(k)^c}\rho_{\ell k}$ and finally on $\sum_{\ell \in S_0(k)} \|h_{\ell k}-h_{\ell k}^0\|_1$. 
To do that we split $[0,T]$ into intervals $[m T/J_T; (m+1)T/ J_T ]$ where  $J_T \leq  \frac{ J_0 T}{ \max( \log T, \log K) } $ with $J_0>0$ possibly small. Then for each $m$ we set, as in \cite{donnet18}, $N^{0,m} $ the process obtained from $N$ where events are considered only if the ancestor (in the sense of the cluster representation of the process) is in $[(2m-1) T/(2J_T); (2m+1)T/( 2 J_T) ]$ and $I_m = [m T/J_T-A; (2m+1)T/( 2 J_T) ]$. Then using Lemma \ref{lem6:donnetsup} in Section \ref{sec:ergodicity} we have for all $H>0$ that there exists $x_0>0$ such that 
 $$ \mathbb P_0( \sum_{m=1}^{J_T}\bar N^m( I_m) > x_0 \log T ) \leq T^{-H}$$
For all $m$ let $t_m \in ( mT/(2J_T) + A, (2m+1)T/(2J_T)- 8A) $, assuming $T$ is large enough and define the event  $\mathcal E_{m,1}(S) = \{  N^{S\cup S_0(k)} [t_m , t_m+2A]=0\}$, then we bound from below 
\begin{align} 
d_{1,T}(f_k, f_k^0) \geq \left(\frac{ 1 }{T} \sum_{m=1}^{J_T} \1_{\mathcal E_{m,1}(S)} \right)2A |\nu_k-\nu_k^0|. 
\end{align}
This implies that for all $z_T>1 $ and on $\bar \Omega_T= \Omega_T\cap \{\sum_{m=1}^{J_T}\bar N^m( I_m) > x_0 \log T\}$
\begin{align*}
\Pi( |\mu_k - \nu_k^0|> z_T v_T |N) &\leq  \sum_{|S|\leq L_T} \1_{\sum_{m=1}^{J_T} \1_{\mathcal E_{m,1}(S)} > T/z_T} \pi(S|N) + o_{P_0}(1)\\
& \leq \max_{|S|\leq L_T} \1_{\sum_{m=1}^{J_T} \1_{\mathcal E_{m,1}(S)} > T/z_T} + o_{P_0}(1)
\end{align*}
Using Lemma \ref{lem:conc:nu} (ii) in Section \ref{sec:additionalLemmas}, we have for all $|S|\leq L_T$, if $z_T \geq 2p_s^{-1} T/J_T$ 
$$\mathbb P_0\left( \sum_{m=1}^{J_T} \1_{\mathcal E_{m,1}(S)} > T/z_T \right) \leq e^{-9p_s  J_T }, \quad p_s = e^{- A(s+s_0)C_1}, \quad C_1=2c_1+(c+3)C_0<  6 C_0$$
So that 
\begin{align*}
\mathbb P_0\left(\max_{|S|\leq s} \sum_{m=1}^{J_T} \1_{\mathcal E_{m,1}(S)} > T/z_T \right) &\lesssim  \sum_{s=1 }^{ L_T} \frac{e^{s\log( K/s)} }{\sqrt{s}}e^{-9p_s  J_T } +\mathbb P_0 ( \bar \Omega_T^c) \\
&\leq L_T e^{ - J_T\left( 9 e^{-AC_1(L_T+s_0)} - \frac{ L_T \log (K)}{J_T}\right) }+ o(1)
\end{align*}
This term is $o(1)$ as soon as $L_T \leq \delta \log T$ with $\delta < 1/(AC_1)$ and 
$$ \log K = o \left( \frac{ J_T e^{-AC_1 (L_T+s_0)}}{ L_T} \right), \quad \text{i.e. } \,  (\log K)^2 = o \left( \frac{ T e^{-AC_1 (L_T+s_0)}}{ L_T} \right)$$
Which then implies that 
\begin{align} \label{rate:nu}
\Pi( |\nu_k - \nu_k^0|> z_T v_T |N) &= o_{P_0}(1), \quad z_T = z_0 e^{AC_1 (L_T+s_0)}T/J_T
\end{align}
Therefore if $s_0$ is bounded and if $L_T = \delta \log T$  then this requires that $(\log K)^2 = o ( e^{-AC_1 s_0} T^{1 - \delta AC_1} /\log T) $ and $z_T = 2T^{\delta AC_1} e^{AC_1 s_0} T/J_T$ 
while if we choose a prior such that $L_T = o( \delta \log \log T)$ or even smaller then this implies that we can choose $(\log K)^3  = o( e^{-AC_1 s_0} T (\log T)^{ - \delta AC_1}/\log \log T) = o(T^{1-\epsilon})$ for all $\epsilon>0$ and $z_T  = 2 e^{AC_1 s_0}  (\log T)^{1+ \delta AC_1}$. Note that in all cases $s_0$ impacts the rate exponentially. 

Using \eqref{rate:nu},  we can bound,  on the event $\mathcal E_{m,2} = \{  N^{S_0(k)} [t_m , t_m+2A]=0\}$ and when $|\nu_k - \nu_k^0|\leq z_T v_T$, 
\begin{align*}
d_{1,T}( f_k , f_k^0) &\geq \sum_{\ell \in  S_0(k)^c} \frac{1}{T} \sum_{m=1}^{J_T-1} \1_{\mathcal E_{m,2}} \left( \int_{t_m}^{t_m+2A}  \int_{t-A}^{t^-} h_{\ell k}(t-s)dN^\ell_s - 2Az_T v_T\right)  dt\\
& \geq  \sum_{\ell \in  S_0(k)^c}\rho_{\ell k} \sum_{m=1}^{J_T-1} \1_{\mathcal E_{m,2}}  N^\ell( t_m, t_m+A) -\frac{  2Az_T v_T J_T}{ T} \\
	& \geq \sum_{\ell \in  S_0(k)^c}\rho_{\ell k} \sum_{m=1}^{J_T-1} \1_{\mathcal E_{m,2}}  \1_{N^\ell( t_m, t_m+A)>0} -\frac{  2Az_T v_T J_T}{ T} .
\end{align*}
Let $X_{m,l}= \1_{ N^{A_0}[t_m, t_m+2A)=0}  \1_{N^{l}( [t_m, t_m+A))>0}$ and
$\tilde X_{m,l} = \1_{ N^{m,0,S_0}[t_m, t_m+2A)=0}  \1_{N^{m,0,l}( [t_m, t_m+A))>0}$
We can then bound on $\bar \Omega_T$ 
\begin{align*}
\max_{l\in S_0(k)^c}\sum_{m=1}^{J_T} |X_{m,l} -\tilde X_{m,l}|  &\leq  \sum_{m=1}^{J_T}\1_{\bar N^m(I_m)>0} \leq  x_0 \log T.
\end{align*}  
Thus using Lemma \ref{lem:conc:nu} (iii) with Hoeffding inequality together with the same argument as for $Z_{m,l}$,  we obtain for all  $u>0$ such that $uJ_T  a_1 >> \log T$,
\begin{align*}
\mathbb P_0 \left( \min_{l\notin S_0(k) }\sum_m (X_{m,l} - E(X_{m,l}) \leq - 2u J_T p(l) \right) \leq K e^{- \frac{  u^2 a_1 J_T}{ 4} } + \mathbb P_0( \bar \Omega_T^c) = o(T^{-H}),
\end{align*}
for all $H$ , since $\log K = o(J_T)$. This leads to 
\begin{equation}\label{conc:rho}
\mathbb E_0 \left[\Pi \left( \sum_{l\in S_0(k)^c} \rho_{lk}> M_1z_Tv_T| N \right)\right]=o(1)
\end{equation}
for some $M_1>0$ . Note that $M_1 $ is exponentially increasing in $s_0$. 

Finally we study for some $M_1'>0$,
$\Pi \left( \sum_{l\in S_0(k)}\|h_{lk}-h_{lk}^0\|_1> M_1'z_Tv_T |N \right). $
When $|\nu_k -\nu_k^0|\leq v_T z_T $ and $\sum_{l\notin S_0(k)}\rho_{lk}\leq M_1z_Tv_T$ then 
\begin{align*}
d_{1,T} (f_k , f_k^0) &\geq \frac{1}{T} \int_0^T| \sum_{l\in S_0(k)}(h_{lk}-h_{lk}^0)(t-s)dN^l_s  -v_T z_T - \frac{1}{T}\sum_{l\notin S_0(k)} \rho_{lk}N^l( -A, T)\\
& \geq \frac{1}{T} \int_0^T| \sum_{l\in S_0(k)}(h_{lk}-h_{lk}^0)(t-s)dN^l_s - 2C_0M_1z_Tv_T.
\end{align*}
We thus need only to bound from below 
$$
I := \frac{1}{T} \int_0^T| \sum_{l\in S_0(k)}(h_{lk}-h_{lk}^0)(t-s)dN^l_s|dt \geq \frac{1}{T} \sum_{m=1}^{J_T} \1_{\mathcal E_{m,3}(l) }\int_{t_m}^{t_m+2A}|h_{lk}-h_{lk}^0|(t-U_m^l)dt,
$$
for all $l\in S_0(k)$ where 
 $$\mathcal E_{m,3} = \{ N^{l}( t_m, t_m+A)=1, \, N^{l}( t_m+A, t_m+2A)=0, \, N^{S_0(k)\setminus \{l\}}(t_m, t_m+2A)=0\}$$
 and $U_m^l$ is the event time of $N^l $ in $(t_m, t_m+A)$.
 On $\mathcal E_{m,3}$ 
 $$\int_{t_m}^{t_m+2A}|h_{lk}-h_{lk}^0|(t-U_m^l)dt = \|h_{lk}-h_{lk}^0\|_1$$
 and using Lemma \ref{lem:conc:nu} (iv) we bound from below 
 $$ q(l) = \mathbb P_0(\mathcal E_{m,3}) \geq a_2 >0$$
 and using the same decompositions as before we bound for all $u_0>0$ 
  $$\mathbb P_0 \left(  \sum_{m=1}^{J_T} (\1_{\mathcal E_{m,3}(l) } - q(l) ) \leq -u_0 J_T \right) \leq e^{-u_0^2 q(l) J_T/4}  + \mathbb P_0( \bar \Omega_T^c)$$
  which terminates   the proof of Theorem \ref{th:gene:L1}. 

\end{proof}




\begin{remark}\label{rk:oracle}
If the connectivity graph (or equivalently $S_0$) were known, then under the Hypotheses (i)-(ii) of Proposition \ref{oracle} on the prior, 
the $L_1$ posterior contraction rate verifies
\begin{equation*}
E_0\left( \Pi_k( \|f_k - f_k^0 \|_1\gtrsim  M_T' e^{6AC_0 s_0}s_0 \e_T \log T |N ) \right) = o(1) .
\end{equation*}
This sharper controls allows to derive a two step procedure described in Section \ref{sec:priors} for which the posterior contraction rate is also of order $e^{6AC_0 s_0}s_0 \e_T$ as soon as $s_0 $ is bounded and $\min\{ \rho_{\ell k}^0, \ell \in S_0(k)\}>> v_T'$ where $v_T' = v_T (\log T)^{6 \delta AC_0} (\log T+ \log K)$. This significantly improves the posterior contraction rate when $\log K \geq  T^{\epsilon}$ for some $\epsilon >0$.
 
\end{remark} 

When the connectivity graph is not known the prior must induce enough sparsity for assumptions (i)-(iii) to be valid and a key quantity to control is $L_T$. To do so we consider a type model selection prior.

\section{ Examples of prior models }\label{sec:priors}


In this section we consider a model selection prior as in Remark \ref{Uniforprior}: 
 $$ |S| \sim \Pi_{1,S} , \quad \Pi_S(S| |S|=s) = \frac{ s! (K-s)!}{K!}, $$
 and we discuss various nonparametric choices for $\Pi_h$ the prior on $h_{\ell k}$ when $\ell \in S$. 

First note that Condition (iii) of Theorem \ref{th:gene:L1} holds for instance if the prior $\Pi_{1,S}$ on $|S|$ is truncated at an arbitrary large constant $\bar s$ or at  $L_0 \log\log T$ or  alternatively  if its tails are of the form
 \begin{equation}\label{tailS}
     \Pi_{1,S}(|S| > x) \leq C\exp( -a_3 e^{ a_2 e^{a_1 x^q }}), \quad q> 1, \quad  C,a_1, a_2, a_3>0 .
 \end{equation}
 
Note that the commonly used Spike and Slab prior defined by 
$ h_{\ell k} \sim p \Pi_h+ (1-p) \delta_{(0)} $, with $\delta_{(0)} $ the Dirac mass at $0$ is a Model selection prior  with a Binomial $\mathcal B( K , p)$ distribution on $|S|$ but it is not a reasonable choice. Indeed 
for Condition (iii) to hold, one needs $p = o(1/K)$. More precisely, if $p \gtrsim 1/K$ then  the tail of $|S|$ decreases at best like $e^{-\lambda x \log x}$ when $Kp \leq \lambda$ or like $e^{- x^2 /kp }$ with $x> Kp$ which implies that $L_T \gtrsim Kp+\sqrt{Kp} \sqrt{T\e_T^2 \log T } >> \log T$. Thus necessarily $p = o(1/K)$ and for each dimension the expected number of non zero interaction functions is $o(1)$. 


While $L_T$ is controlled by the sparsity structure of the prior distribution, $\e_T$ is typically a consequence of the smoothness of the functions $h_{l k }^0$, $l\in S_0(k)$ and of properties of the prior distributions $\Pi_h$ so that assumptions  (i)-(ii) hold. These assumptions are standard in the Bayesian nonparametric literature aside from the constraints

\subsection{Examples of priors $\Pi_h$ and implications on $\e_T$}

The prior marginal  distributions  $\Pi_h$ defined in \eqref{model:prior} are characterized by parametric representations of $h$ (resp. $\bar h$) with infinite or large number of parameters. We present a number of priors and their associated $\e_T$, which have been derived in \cite{donnet18}. 

\begin{itemize}
\item Random histogram prior : Let $0=t_0<t_1< \cdots, t_I=A$, then $h \sim \Pi_h$ if 
\begin{equation}\label{histo}
h(x)= \sum_{ i =0}^{I-1} h_i \1_{x\in (t_i, t_{i+1}]},\quad h_i \geq 0, \quad \sum_{i}h_i(t_{i+1}-t_i)<1
\end{equation}
$\Pi_h$ is then obtained by considering a prior distribution  on $I,( h_0, \cdots, h_{I-1})$. For instance $I \sim \mathcal P(a)$ for some $a>0$ and given $I$, set $(w_1, \cdots, w_{I+1}) \sim \mathcal D(\alpha, \cdots, \alpha)$ and define for $0\leq i \leq I-1$ $h_i = w_{i+1}/(t_{i+1}-t_i)$ . 

\begin{corollary}\label{coro:histo}
Assume that for all $\ell \in S_0(k)$, $h_{\ell k}^0 $ is $\beta $  Holder with $\beta \leq 1$. Then the random histogram  prior $\Pi_h$ defined above implies that 
$\e_T = (T/\log T)^{-\beta/(2\beta+)}$. Then under the assumptions of Theorems \ref{th:gene:d1T} and \ref{th:gene:L1} on $K$ and under Hypotheses \ref{cond:suprho}-\ref{cond:sparsity2}, together with $L_T \leq \delta \log \log T$ the $d_{1,T}$ and the $L_1$ posterior contraction rates are bounded by 
 $$  M_T (\log T)^q \max( T^{-\beta /(2\beta+1)} , \sqrt{ \log K/T} )$$
 \end{corollary}

\item log-spline  priors:  
We have  $h \sim \Pi_h$ if  
\begin{equation}\label{spline}
\begin{split}
h(x) &= e^{\sum_{j=1}^J \theta_j B_j(x)}  , \quad  J \sim \mathcal P(a)\\  
 \theta_j , j\leq J &\stackrel{ iid }{\sim} \mathcal N(0, \tau^2), \quad \tau >0
\end{split}
\end{equation}
where $(B_1, \cdots, B_J)$ is a B-spline basis of order $m$ associated to the regular grid $(0, iA/I, i\leq I) $, see for instance Section 4 of \cite{ghosal00}. Then a direct adaptation of the results of Section 4 of \cite{ghosal00} leads to 
\begin{corollary}\label{coro:spline}
Assume that for all $\ell \in S_0(k)$, $h_{\ell k}^0 $ is $\beta $  Holder with $m +1 \leq \beta>0$. Then the log-spline  prior $\Pi_h$ defined in \eqref{spline} implies that 
$\e_T = (T/\log T)^{-\beta/(2\beta+)}$. Then under the assumptions of Theorems \ref{th:gene:d1T} and \ref{th:gene:L1} on $K$ and under Hypotheses \ref{cond:suprho}-\ref{cond:sparsity2}, together with $L_T \leq \delta \log \log T$ the $d_{1,T}$ and the $L_1$ posterior contraction rates are bounded by 
 $$  M_T (\log T)^q \max( T^{-\beta /(2\beta+1)} , \sqrt{ \log K/T} )$$
 \end{corollary}



 \end{itemize}
 

Many other types of priors could be considered, such as log - Gaussian process priors, mixtures of rescaled Beta priors , for which $\e_T$ have been derived either in the context of Hawkes processes (for the latter , in\cite{donnet18}) or in the context of density  or nonparametric regression estimation. 

In the following section we suggest a two step procedures which allow to derive a rate which only depends on $\e_T$ and $s_0$ as soon as $\log K = o(\sqrt{T})$.  

\subsubsection{ A two step procedure} \label{sec:two-steps}
Under the conditions of Theorem \ref{th:gene:L1} with $L_T\leq \delta \log \log T$ the following two step procedure then allows to improve significantly the posterior contraction rate when $K$ is very large: For all $k\in [K]$
 \begin{itemize}
 \item Step 1: 
 \begin{itemize}
     \item (i) Implement the full posterior distribution  $\Pi_k (\cdot |N) $  and compute 
 $\hat \rho_k = (\mathbb E^{\pi}( \rho_{\ell k} | N), \ell\in [K])$ 
 \item (ii) For all $k$, rank $\hat \rho_{\ell k}$ in decreasing order and denote it $\tilde \rho_k  = (\hat \rho_{\sigma(\ell) k}, \ell \in [K])$, choose a threshold level $u_T = o(1) $ and define $\hat S(k)= \{ \ell ; \sigma(\ell) < \hat \ell \} $ where $\hat \ell = \min \{\ell ; \, \sum_{j \geq \ell} \hat \rho_{\sigma(j) k} \leq u_T\} $.
 \end{itemize}
 \item Step 2: Implement the posterior distribution  conditional on $S(k)= \hat S(k)$, denoted $\Pi_{2,k} ( \cdot | N) = \Pi_k( \cdot | N, S(k) = \hat S(k))$.
 \end{itemize}
 Then we have the following proposition,
 
 \begin{proposition}\label{prop:2steps}
 Under the conditions of Theorem \ref{th:gene:L1} with $L_T \leq \delta \log\log T$ for some $\delta >0$ and if in addition to (iii)  with a selection prior as in Remark \ref{Uniforprior} the prior  $\Pi_{1,S}$ on $|S|$ also verifies
 \begin{equation}\label{priorP1_plus}
 \sum_{s>  L_T} s \Pi_{1,S}(s) \leq e^{- M_T s_0^2 T \e_T^2 \log(\log T + s_0)- M_T s_0 \log K}
 \end{equation}
 choosing $o(1) = u_T >> v_T' = v_0'(\log T)^{q}(1+ \log K)(\e_T + \sqrt{\log K/T})$, the rate obtained in Theorem \ref{th:gene:L1}  and as soon as $\inf \{ \rho_{\ell k}^0; \ell \in S_0(k)\} \geq  2u_T$,
  \begin{align}\label{Set-estimation}
  \mathbb P_0 \left( \hat S(k) \neq S_0(k)\right) = o(1)
  \end{align}
  and  for some $M_T'$ going to infinity, 
  \begin{equation}\label{step2rate}
     \mathbb E_0[\Pi_{2,k}( \| f_k -f_k^0\|_1 \leq  M_T's_0 \e_T e^{ 6As_0 C_0} \log T|N)]  = 1 + o(1)
  \end{equation}
  \end{proposition}
 \begin{proof}[Proof of Proposition \ref{prop:2steps}]
 We first prove \eqref{Set-estimation}.
  Note that by convexity and boundedness of $\rho_{\ell k}\leq 1$
  \begin{align*}
       \|\hat \rho_k - \rho_k^0\|_1 &\leq E^\pi( \|\rho_k - \rho_k^0\|_1 |N ) \leq v_T' +E^\pi(\1_{\|\rho_k - \rho_k^0\|_1>v_T'} (|S|+s_0)|N)  
   \end{align*}
   since $T\e_T^2 \geq (\log T)^3$.
 Also $\Pi( \|\rho_k - \rho_k^0\|_1>v_T'|N) \leq \Pi( \|f_k -f_k^0\|_1>v_T'|N)$, so that 
 from the Proof of Theorem \ref{th:gene:L1}, on an event of probability going to 1,  $\Pi( \|f_k -f_k^0\|_1>v_T'|N) \leq  \Pi( d_{1,T}(f_k, f_k^0)>v_T|N) + \Pi( |S(k)| > L_T|N)$ where $v_T'$ is the rate obtained in Theorem \ref{th:gene:L1} and $v_T$ that of Theorem \ref{th:gene:d1T}. Moreover, on an event of probability going to 1 
  $$ \Pi( d_{1,T}(f_k, f_k^0)>v_T|N) + \Pi( |S|> L_T|N) \leq e^{ - M_T' s_0 ( s_0T\e_T^2 \log \log T + \log K)} $$
  on $\Omega_T \cap \{ D_{T,k} \geq e^{- z_T s_0( s_0T\e_T^2 \log \log T + \log K)}\}$ which has probability going to 1 for any $z_T$ going to infinity 
  \begin{align*}
  \sum_{s>L_T} s\Pi( |S(k)|=s|N) &\leq e^{ z_T s_0( s_0T\e_T^2 \log \log T + \log K)} \sum_{s>L_T}s \Pi_{1,S}(s)\leq 
  e^{ - M_Ts_0( s_0T\e_T^2 \log \log T + \log K)/2} 
  \end{align*}
  for some $M_T'$ going to infinity. Finally  
 on an event of probability $1 + o(1)$,
  \begin{align*}
       \|\hat \rho_k - \rho_k^0\|_1 \leq 2 v_T' 
   \end{align*}
 where $v_T' \lesssim  (\log T)^q (s_0\e_T + \sqrt{ s_0 \log K/T})$ for some $q>0$. 
 Note that if $S_0(k) \cap \hat S(k)^c \neq \emptyset$ then there exists $ \ell \in S_0(k)$ such that $\hat \rho_{\ell k}\leq u_T$ which in turns implies that $|\hat \rho_{\ell k}- \rho_{\ell k}^0|> u_T $ and $\|\hat \rho_k - \rho_k^0\|_1\geq u_T$ which has probability going to 0. Similarly if $S_0(k)^c \cap \hat S(k) \neq \emptyset$, then there exists $\ell $ such that $\hat \rho_{\ell k} > u_T$ and $\rho_{\ell k}^0 = 0$ which also implies that $\|\hat \rho_k - \rho_k^0\|_1> u_T$. This proves \eqref{Set-estimation}.

 On the event $\hat S(k) = S_0(k)$ the posterior distribution given $S(k) = \hat S(k)$ is the posterior distribution given $S(k) = S_0(k)$, whose posterior contraction rate is given by  $M_T's_0 \e_T e^{ 6As_0 C_0}$, as seen in Remark \ref{rk:oracle}. This terminates the proof of Proposition \ref{prop:2steps}.
  \end{proof}

Under the priors considered in Section \ref{sec:priors} and if $s_0$ is bounded and the functions $h_{\ell k}^0$ are $\beta $ Holder, then the $L_1$ posterior contraction rate of the two - step procedure is 
  $v_T'= M_T'T^{-\beta/(2\beta+1)} (\log T)^{1+ \beta /(2\beta+1)}$, as soon  as $\log K = o ( \min (T^{1/3-\epsilon}, T^{(\beta-\epsilon)/(2\beta+1)} )$ for any $\epsilon>0$.

\section{discussion} \label{sec:discussion}
As mentioned in Section \ref{sec:Intro}, understanding how sparsity affects the hardness of the estimation problem is not trivial and there are no results so far on the fundamental limits of this problem. In our paper we have tried to derive results under rather simpler sparsity patterns and while we believe that our constraints on the sparsity to derive rates in the empirical loss function $d_{1,T}$ are close to being optimal, it is not at all clear that the constraint $L_T< \delta \log \log T$ (and thus $s_0 $) required for the convergence in $L_1$ is necessary. To shed more light on the role of sparsity in high dimensional Hawkes processes, we discuss in this section, existing results and assumptions on high-dimensional Hawkes processes. We recall that all previous results hold for frequentist approaches (based on penalised least-square estimators) and no results have been previously obtained for Bayesian methods.
 
In the nonparametric setting and for linear Hawkes processes, \cite{Chen2017} propose an interaction graph estimator $\hat S$ which contains the true graph with high probability, under minimal assumptions of stationarity and boundedness on the parameter and if the (known) sparsity $s_0$  and the dimension $K$ are such that $s_0^{1/2}K^2 = o(e^{T^{1/6}})$. They also require a fixed lower bound on the $L_2$ (instead of the $L_1$) norm of the non zero interaction functions, which means that in order to verify their stationary assumption they also need $s_0 $ to be bounded. Moreover, unlike our setting, this procedure does not allow to recover neither the ground-truth intensity nor parameter.
 
For linear nonparametric Hawkes processes, under the condition $s_0 \log K = o(T)$, \cite{bacry-etal-2020} propose a penalized least-square estimate whose $l_2$-stochastic risk is controlled by $s\log K$ up to some random constants. However,  the constants entering these bounds are not explicitly controlled in the regime $K \to \infty$.
 
The results of \cite{Tang-Li-2023} also provide insight into high-dimensional nonparametric Hawkes processes. Their approach is original in that it uses a tensor representation and then imposes subgroup structures on the coefficient tensor. The sparsity constraint is expressed via $R$, the rank of the tensor of function-basis coefficients. Then, the authors obtain rates proportional to $\sqrt{R\log T/T}$. 
Under our sparsity pattern $R$ would correspond to $L_T K J$ where $J$ is the number of basis functions used to represent $h_{\ell k}$, which leads to a suboptimal (and non adaptive ) rate. Note however that the framework of
\cite{Tang-Li-2023} allows for more general sparsity pattern is not so well suited for ours. They also assume
a lower bound on the eigenvalues of large matrices. This condition seems  much stricter than ours and not so easy to meet.
 
For nonlinear Hawkes processes, \cite{chen2019} obtained a convergence rate for the cross-covariance estimates of order $O(T^{-2/5})$  under the assumption $\|\rho\|_\infty<1$ and the intensity is  bounded. No theoretical guarantees are provided for the estimation of parameters of the model. \cite{cai2024latent} use penalised B-splines and prove that the $l_2$-stochastic risk is of order $O(s\lambda_{max} \log K)$, under the assumptions that $\|\rho^T \rho\|< 1$, $\|\rho\|_\infty < 1$ the intensity is bounded by $\lambda_{max}$, and $s$ and $K$ are such that $s=o(T^{2/5}), \log K = O(T^{1/5})$. Observe that assuming that the intensity is bounded is a  strong assumption and does not allow to consider the ReLU link function or simple linear Hawkes processes.
 
In a parametric linear Hawkes model including covariates, \cite{kreiss2025commondriverssparselyinteracting}  propose a debiased LASSO-type estimate with convergence rate of order $O(s\log (KT)/\sqrt{T})$ in a regime where both $s$ and $K$ are dominated by a power of $T$ and $\|\rho\|_1 < 1$. Several of their other assumptions such as the random compatibility condition are not explicit and difficult to verify. Also in the parametric setting, \cite{wang2025statistical} design a testing procedure for detecting edges in a high-dimensional Hawkes graph. Their assumptions is similar to \cite{cai2024latent} with the additions that $s^4 \vee \log K =o(T^{1/5})$, $\|\rho^T\|_\infty< \infty,$ and $\|\rho\|_\infty < 1$.

It thus appears that under a sparsity pattern defined in terms of the number of non zero interaction functions per dimensions and in terms of the operator norms of $\rho^0$, our results allow to control the convergence rate either in the weaker $d_{1,T}$ loss  or in the stronger  $L_1$ loss under conditions on $K$ which are weaker than the existing results. Our results  under the $d_{1,T}$ loss also allow for weak assumptions on the sparsity and lead to adaptive (with respect to the smoothness of the functions) nonparametric convergence rates. Under much stronger assumptions on the sparsity $s_0$, these results can also be transferred to the $L_1$ loss. Note also that the proof which allow to bound from bellow the distance $d_{1,T}$ by the distance $L_1$ can be used in non Bayesian settings and could be adapted to some of the results derive in the papers described above.

\section{Appendix} \label{sec:proofs}
We include here Lemma \ref{lem:conc:nu} together with Lemma \ref{lem6:donnetsup} which are crucial to derive $L_1 $ posterior contraction rates from $d_{1,T}$ posterior contraction rates.

\subsection{Lemma \ref{lem:conc:nu} used in the Proof of Theorem \ref{th:gene:L1}} \label{sec:additionalLemmas}
\begin{lemma}\label{lem:conc:nu}
Under the assumptions of Theorem \ref{th:gene:d1T} with  $L_T \leq \delta \log T$ and $0 <\delta <1/(6 A C_0)$, let $B\geq 2A$ and 
  then  for all $S\subset [K]$ such that $|S| =s \leq L_T +s_0	$
 \begin{itemize}
 \item (i)   $\mathbb P_0( N^S [0,B]=0 ) \geq e^{-s)[Bc_1+AcC_0  +(B+A) C_0] }:= p_s(B), $
 \item (ii)  
 \begin{align}
 \mathbb P_0\left( \sum_{m=1}^{J_T-1} \1_{\mathcal E_{m,1}(S)} > T/z_T \right) \leq e^{ - 9 J_T e^{-(s+s_0)A C_1 }  }   \quad z_T = \frac{2 T}{ J_Tp_{s+s_0} (2A)} 
   \end{align}
   \item (iii) for all $k$ and all $l \in S_0(k)^c$ , let $a_0 = 1 - e^{-c_0A}$,
$$ p(l) = \mathbb P_0 [ N^{S_0(k)}[t_m, t_m+2A)=0 \& N^{l}[t_m, t_m+A)>0] \geq a_1>0$$  with $a_1= (1 - e^{-c_0A})e^{-As_0(2c_1 +(c+1+2/a_0)C_0)}$
 \end{itemize}

\end{lemma}
\begin{proof}[Proof of Lemma \ref{lem:conc:nu}]

Let $S \subset [K]$,  write $\tilde S = S \cup \tilde S_0(k)$ and  $\tilde Z_{m,S}= 1_{N^{\tilde S, 0 , m}[t_m , t_m + 2A)=0} - \mathbb P_0 ( N^{\tilde S, 0 , m}[t_m , t_m + 2A]=0 $, then 
\begin{align*}
 \sum_{m=1}^{J_T-1} \1_{\mathcal E_{m,1}(S)} & =  \sum_{m=1}^{J_T-1} \1_{N^{\tilde S}[t_m , t_m + 2A)=0}\\
 &  \geq ( J_T-1) p_s + \sum_{m=1}^{J_T-1} \tilde Z_{m,S} - \sum_{m=1}^{J_T}  \1_{N^{\tilde S, 0 , m}[t_m , t_m + 2A)=0} \1_{N^{\tilde S, 0 , m}[t_m , t_m + 2A)>0}\\
 & \geq  ( J_T-1) p_s + \sum_{m=1}^{J_T-1} \tilde Z_{m,S} - \sum_{m=1}^{J_T}\1_{\bar N^m( I_m) >0}
 \end{align*}
 since $N^{\tilde S, 0 , m}[t_m , t_m + 2A]\leq N^{\tilde S}[t_m , t_m + 2A]$. 
 On $\bar \Omega_T$, $\sum_{m=1}^{J_T}\1_{\bar N^m( I_m) >0}\leq x_0 \log T$ and 
we now bound from below 
$$p_S = \mathbb E_{\mathbb Q_{f_0; t_m}}\left[ \mathcal L_{[t_m,t_m+B)}(f_0 )\1_{N^{\tilde S}[t_m,t_m+B)=0}\right] $$
where $ \mathbb Q_{f_0; t_m}$ is the distribution $ \mathbb P_{f_0}$ over $(-\infty , t_m)$ and an independent product of homogeneous PP on $[t_m, t_m+ B)$. For the sake of simplicity and without loss of generality (by stationarity) we consider below $t_m=0$.
\begin{align*}
\mathcal L_{[0,B]}(f_0) &= \mathcal L_{[0,B]}(f_{\tilde S^c}^0) \mathcal L_{[0,B]}(f_{\tilde S}^0) \\
& = \mathcal L_{[0,B]}(f_{\tilde S^c}^0) e^{B(|\tilde S|-\nu_{\tilde S}^0) - \sum_{l\in {\tilde S}} \sum_{l'} \int_0^B\int_{t-A}^{t-}h_{l'l}^0(t-s)dN^{l'}_s dt }\\
& \geq \mathcal L_{[0,B]}(f_{{\tilde S}^c}^0) e^{B(|{\tilde S}|-\nu_{\tilde S}^0) - \sum_{l\in {\tilde S}} \sum_{l'}\rho_{l'l}^0N^{l'}(-A,0)- \sum_{l\in {\tilde S}} \sum_{l'\in {\tilde S}^c}\rho_{l'l}^0N^{l'}(0,B)  }
\end{align*} 
 with $\nu_{\tilde S}^0 = \sum_{l \in {\tilde S}} \nu_l^0$ and where we have used $N^{\tilde S} (0,B)=0$. 
We also have that 
\begin{align*}
\mathcal L_{[0,B]}(f_{{\tilde S}^c}^0) & = 
\prod_{l\in \tilde S^c}  \prod_{t \in N^l(0,B]}\lambda_t( f_l^0 )
e^{(1- \nu_l^0)B - \int_0^B \sum_{l'\tilde S^c} \int_{t-A}^{t^-}h_{l'l}^0(t-s)dN^{l'}_sdt - \int_0^A \sum_{l' \in \tilde S} \int_{t-A}^{t^-\wedge 0}h_{l'l}^0(t-s)dN^{l'}_sdt } \\
& \geq 
\mathcal L_{[0,B]}(\tilde f_{{\tilde S}^c}^0) e^{- \sum_{l\in \tilde S^c}\int_0^A \sum_{l' \in \tilde S} \int_{t-A}^{t^-\wedge 0}h_{l'l}^0(t-s)dN^{l'}_sdt - \int_0^B \sum_{l'\tilde S^c} \int_{t-A}^{t^-}h_{l'l}^0(t-s)dN^{l'}_sdt }
\end{align*}  
where $\tilde f_{{\tilde S}^c}^0 = (\nu_l^0; h_{l'l}^0, l',l\in \tilde S^c) $ in other words it has  forced $h_{l'l}^0=0$ for $l'\in \tilde S$. 
Using 
 $$ \int_0^A\int_{t-A}^{t^-\wedge 0}h_{l'l}(t-s)dN^{l'}_sdt \leq \rho_{l'l} N^{l'}(-A,0), \quad  \int_0^B\int_{t-A}^{t^-}h_{l'l}(t-s)dN^{l'}_sdt \leq \rho_{l'l} N^{l'}(-A,B), $$
 we obtain 
 \begin{align*}
\mathbb Q_{f_0}&\left[ \mathcal L_{[0,B]}(f_0 )\1_{N^{\tilde S}(0,B)=0}\right] 
 \geq e^{-\nu_{\tilde S}^0B}
\mathbb E_{\mathbb P_0}\left[ 
	\mathbb E_{P_{\tilde f_{\tilde S^c}^0}} \left( \left.e^{- \sum_{l\in \tilde S^c} \sum_{\tilde S} \rho_{l'l}N^{l'}(-A,0)  - \sum_{l\in S}\sum_{l'\notin \tilde S}\rho_{l'l}^0N^{l'}(-A,B) } \right|\mathcal G_0 \right) \right] \\
	&\geq  e^{-\nu_{\tilde S}^0B}e^{- \sum_{l\in \tilde S^c} \sum_{l'\in \tilde S} \rho_{l'l}E_{\mathbb P_0}(N^{l'}(-A,0) ) - \sum_{l\in S}\sum_{l'\notin \tilde S}\rho_{l'l}^0 E_{\mathbb P_0}( N^{l'}(-A,B)) } \\
	& \geq  e^{-\nu_{\tilde S}^0B}e^{-A \sum_{l\in \tilde S^c} \sum_{l'\in \tilde S} \rho_{l'l}\mu_{l'}^0  -(B+A) \sum_{l\in S}\sum_{l'\notin \tilde S}\rho_{l'l}^0 \mu_{l'}^0  }\\
	& \geq & \geq  e^{-(s+s_0)[Bc_1+AcC_0  +(B+A) C_0]},
 \end{align*} 
 which proves (i) and where we have used   $\max_{l'}\sum_l \rho_{l'l}^0\leq c<1$. 
  Finally using Bernstein inequality for Bernoulli variables with $\tilde p_S = \mathbb P_0(N^{\tilde S, 0, m}[t_m,t_m+2A]=0) $  and choosing $B = 2A$
 \begin{align*}
\mathbb P_0 ( |\sum_m \tilde Z_{m,S} | > x )  & \leq e^{-\frac{ x^2 }{ J_T \tilde p_S(1-\tilde p_S) + x/3} } \leq e^{ - \frac{ T^2 }{ 2z_T^2J_Tp_{s+s_0}(2A)  + Tz_T/3} }	\\
&\leq e^{ - \frac{ 36 J_T^2 p_{s+s_0}^2 }{4 J_Tp_{s+s_0}(2A) } }\leq e^{ - 9 J_T e^{-(s+s_0)A C_1 }  } 
\end{align*}
by choosing $x = T/z_T = 6 J_T p_{s+s_0}$  
and noting that 
$$\tilde p_S  \leq P_0(N^{\tilde S}[0,2A]=0)+\mathbb P_0( \bar N^m ( I_m)=0) \leq p_{s+s_0}(2A) +\frac{  x_0 \log T }{ J_T} \leq 2 p_{s+s_0}(2A) $$
as soon as $e^{-(s+s_0)A C_1 } > x_0(\log T)^2 /(J_0T)$, i.e. $L_T  < \delta \log T$ with $\delta \leq  1/(AC_1) - \epsilon$  for some $\epsilon>0$,
 which proves (ii). 
 We now prove (iii) . Define, for $l\notin S_0(k)$,  $p(l) = E_0( X^{m,l})= E_{\mathbb Q_{f_0}}\left[ \mathcal L_{[0,2A]}(f_0) \1_{ N^{S_0(k)}[0,2A)=0} \1_{N^{l}( [0, A))> 0} \right] $, by stationarity. 
When $N^{S_0(k)}[0,B)=0$,
\begin{align*}
\mathcal L_{[0,2A]}(f_0)  & = \mathcal L_{[0,2A]}(f_{ S_0(k)^c}) e^{2A(s_0-\nu_{S_0(k)}^0) - \sum_{j\in {S_0(k)}} \sum_{l'\in [K]} \int_0^{2A}\int_{t-A}^{t-}h_{l'j}(t-s)dN^{l'}_s dt } \\
& \geq \mathcal L_{[0,2A]}(f_{S_0^c(k)}) e^{2A(s_0-\nu_{S_0(k)}^0) - \sum_{j\in S_0(k)} \sum_{l'}\rho_{l'j}^0N^{l'}(-A,0)- \sum_{j\in S_0(k)} \sum_{l'\in S_0(k)^c}\rho_{l'j}^0N^{l'}(0,2A)  } 
\end{align*}
Also 
\begin{align*}
\mathcal L_{[0,2A]}(f_{S_0(k)^c}) &= \prod_{j \in S_0(k)^c} \prod_{T \in N^j_{[0,2A)}} \lambda_T(f_j^0) e^{(1- \nu_j)2A- \int_0^{2A} \sum_{l'\in S_0(k)^c} \int_{t-A}^{t^-}h_{l'j}(t-s)dN^{l'}_sdt - \int_0^A \sum_{l' \in S_0(k)} \int_{t-A}^{t^-\wedge 0}h_{l'j}(t-s)dN^{l'}_sdt } \\
& \geq 
\mathcal L_{[0,2A]}(\tilde f_{S_0(k)^c}^0) e^{- \sum_{j\in S_0(k)^c}\int_0^A \sum_{l' \in S_0(k)} \int_{t-A}^{t^-\wedge 0}h_{l'j}(t-s)dN^{l'}_sdt } \\
& \geq \mathcal L_{[0,2A]}(\tilde f_{S_0(k)^c}^0) e^{- \sum_{j\in S_0(k)^c} \sum_{l' \in S_0(k)} \rho^0_{l'j}N^{l'}[-A,0) }
\end{align*}
where similarly to before, $\tilde f_{S_0(k)^c}^0  = (\nu_l^0, h_{l'l}^0, l', l \in A_0^c)$ so that 
\begin{align*}
p(l) &\geq  e^{-2A\nu_{S_0(k)}^0} E_{P_{f_0}}\left[ e^{- \sum_{j\in S_0(k)^c} \sum_{l' \in S_0(k)} \rho^0_{l'j}N^{l'}[-A,0) }e^{ - \sum_{j\in S_0(k)} \sum_{l'}\rho_{l'j}^0N^{l'}(-A,0)} \times \right. \\
 & \quad \left.  E_{P_{\tilde f_0}}\left[  \left.e^{- \sum_{j\in S_0(k)} \sum_{l'\in S_0(k)^c}\rho_{l'j}^0N^{l'}(0,2Q)  } \1_{N^{l}( [0, A))> 0}  \right| \mathcal G_0\right] \right]
\end{align*} 
Now 
\begin{align*}
\mathbb E_{P_{\tilde f_0}}& \left[  \left.e^{- \sum_{j \in S_0(k)} \sum_{l'\in S_0(k)^c}\rho_{l'j}^0N^{l'}(0,B)  } \1_{N^{l}( [0, A))> 0}  \right| \mathcal G_0\right]\\
& \geq 
\mathbb P_{\tilde f_0} \left[ \left. N^{l}( [0, A))> 0 \right| \mathcal G_0\right] \exp\left( - 
 \sum_{j \in S_0(k)} \sum_{l'\in S_0(k)^c}\rho_{l'j}^0\mathbb E_{P_{\tilde f_0}}( N^{l'}(0,2A) | \mathcal G_0 ,  N^{l}( [0, A))> 0 )\right)
\end{align*} 
First note that 
$$ \mathbb P_{\tilde f_0} \left[ \left. N^{l}( [0, A))> 0 \right| \mathcal G_0\right] \geq 1 - e^{-\nu_l^0 A} \geq 1 - e^{-c_0A} := a_0$$
so that we also have that 
$$ E_{P_{\tilde f_0}}( N^{l'}(0,2A) | \mathcal G_0 ,  N^{l}( [0,A))> 0 ) \leq 
\frac{  E_{P_{\tilde f_0}}( N^{l'}(0,2A) | \mathcal G_0  )}{ 1 - e^{-\nu_l^0A} }\leq  \frac{  \mathbb E_{P_{ f_0}}( N^{l'}(0,2A) | \mathcal G_0  )}{ 1 - e^{-\nu_l^0A} }$$
which in turns leads to 
\begin{align*}
p(l) &\geq a_0 e^{-2A\nu_{S_0(k)}^0} E_{P_{f_0}}\left[ e^{- \sum_{j\in S_0(k)^c} \sum_{l' \in S_0(k)} \rho^0_{l'j}N^{l'}[-A,0) }e^{ - \sum_{j\in S_0(k)} \sum_{l'}\rho_{l'j}^0N^{l'}(-A,0)} \times \right. \\
 & \quad \left. \exp\left( - 
 \sum_{j \in S_0(k)} \sum_{l'\in S_0(k)^c}\frac{ \rho_{l'j}^0}{ a_0 }\mathbb E_{P_{\tilde f_0}}( N^{l'}(0,2A) | \mathcal G_0  )\right)\right]\\
 & \geq  a_0 e^{-2A\nu_{S_0(k)}^0} \exp\left[ - A \sum_{j\in S_0(k)^c} \sum_{l' \in S_0(k)} \rho^0_{l'j}\mu_{l'}^0-A \sum_{j\in S_0(k)} \sum_{l'}\rho_{l'j}^0\mu_{l'}^0-2A\sum_{j\in S_0(k)} \sum_{l'\in S_0(k)^c}\frac{\rho_{l'j}^0}{ a_0} \mu_{l'}^0 \right]\\
 & \geq a_0 e^{-\nu_{S_0(k)}^0} \exp\left[ - A c s_0C_0-(A + 2A/a_0)s_0C_0\right] \geq  a_1.
\end{align*}
 \end{proof}

\subsection{Ergodicity } \label{sec:ergodicity}

\subsubsection{Lemma \ref{lem6:donnetsup}}
In this section we control the erodicity of the process, using a construction, as in \cite{donnet18}.  Let $(H_t(N))_t$ a predictable process, let $J_T = \lfloor  J_0 T/ \log T\rfloor $  and for all $m \leq J_T$ define $N^{0,m}$ the subprocess of $N$ whose ancestors are born in $[(2m-1)T/(2J_T), (2m+1)T/(2J_T)]$. Note that the subprocesses $ N^{0,m} $ are independent and identically distributed. 

Let $I_m = [2mT/(2J_T) -A, (2m+1)T/(2J_T)]$ and $\bar N^m(I_m)  = N(I_m) - N^{0,m}(I_m)$
then we have the following refinement of Lemma 6 in the supplement of \cite{donnet18:supplement}:
\begin{lemma}\label{lem6:donnetsup}
Under Hypotheses \ref{cond:suprho}-\ref{cond:sparsity2} and if $J_T \leq  \frac{ Tt(c) }{ 3\log ( \|\nu^0\|_1  )} \asymp \frac{T}{\log K}$ with $t(c) = c-1 - \log c$,  then for all $x >0$  
$$\mathbb P_0( \sum_{m=1}^{J_T} \bar N^m(I_m) > x T )\leq 2 e^{ - t(c) x T/4} .$$
In particular, choosing $x = x_0 \log T/T$ we obtain 
$$\mathbb P_0( \sum_{m=1}^{J_T} \bar N^m(I_m) > x_0\log T )\leq 2 T^{-x_0t(c)/4}$$
\end{lemma}
\begin{proof}[Proof of Lemma \ref{lem6:donnetsup}]
We follow the proof of Lemma 6 of the supplement of \cite{donnet18:supplement}:
by construction $\bar N^m$ only contains points whose ancestors are born before $(2m-1)T/(2J_T)$ so that the distance between an occurence of an ancestor in $\bar N^m$ and $I_m$ is at least $T/(2J_T)-A$. Using the cluster representation of the process and as in \cite{donnet18}, 
\begin{align*}
\sum_{m=1}^{J_T} \bar N^m(I_m)& \leq \sum_{k=1}^K \sum_{p=1}^{T - \lfloor T/(2J_T)\rfloor} \sum_{i=1}^{B_{p,k}} \left( W_{p,i}^k-\frac{ 1}{A} \left( \frac{ T}{2J_T}-A\right)\right)_+ \\
 & \quad + \sum_{k=1}^K \sum_{p=-\infty}^{0} \sum_{i=1}^{B_{p,k}} \left( W_{p,i}^k-\frac{ 1}{A} \left( -p -1  + \frac{ T}{2J_T}-A\right)\right)_+\\
  & := T_1+ T_2
\end{align*}
where the $B_{p;k}$ are independent Poisson $\mathcal P(\nu_k^0)$ and $W_{p,i}^k$ is the number of points generated by ancestor $i$. For each $k$ $(W_{p,i}^k)_i$ are iid with the same distribution as $W^k$. We therefore have, using Lemma \ref{lemW}  for all $k$
$$ \mathbb E_0( e^{t W^k} ) \leq e^{x(t,c)} , \quad \forall t \leq c-1-\log(c):= t(c)$$ 
where $x(t,c)$ is the solution smaller than $\log(1/c)$ of the equation 
$ x -c(e^x-1)=t$.
Therefore for all $x>0$ and $t < t(c)$
\begin{align*}
\mathbb P_0( T_1 > xT ) &\leq e^{-t x T} \prod_{k=1}^K \prod_{p=1}^{T - \lfloor T/(2J_T)\rfloor} \mathbb E_0(H_k(t)^{B_{p,k}} ), \quad H_k(t) = \mathbb E_0\left( e^{t \left( W^k-\frac{ 1}{A} \left( \frac{ T}{2J_T}-A\right)\right)_+ } \right)\\
& \leq  e^{-t x T +  \left(T - \lfloor T/(2J_T)\rfloor\right)\sum_{k=1}^K \nu_k^0(H_k(t)-1)}
\end{align*}
Moreover choosing $t=t(c)$
\begin{align*}
H_k(t) &\leq 1 + \frac{e^{ - t(c)[T/(2J_T)-A]}}{c}, \quad 
 \sum_k\nu_k^0( H_k(t) -1)& \leq \frac{ e^{t(c)A} \|\nu^0\|_1 e^{ - t(c)T/(2J_T)} }{c} =o(1)
\end{align*}
as soon as $ T/J_T\geq 2\log ( \|\nu^0\|_1  )/t(c)$. 
Then
choosing  $x \gtrsim 1/T$ 
\begin{align*}
\mathbb P_0( T_1 > xT ) 
& \leq  e^{-t(c) x T /2}  
\end{align*}
Replacing for $p\leq 0$,   $H_k(t) $ by 
\begin{align*}
 \mathbb E_0\left( e^{t \left( W^k-p-1 - \frac{ 1}{A} \left( \frac{ T}{2J_T}-A\right)\right)_+ } \right) & \leq 1 + \frac{e^{  t(c)( A+ p+ 1) } e^{ - t(c)T/(2J_T)} }{c},
 \end{align*}
 leads to 
 $$P( T_2 > x T) \leq e^{-t(c)xT}e^{\frac{  \|\nu^0\|_1e^{ - t(c)T/(2J_T)} e^{  t(c)( A+1)}    }{c (1-e^{-t(c)}) }} \leq e^{-t(c) x T/2}. $$
\end{proof}

Lemma \ref{lem6:donnetsup} is crucial and is used in particular to prove that uniformly in $k$ $|N^k[0,T]/T -\mu_k^0| \leq \delta_T= o(1)$ with high probability. 
Denote by 
 $$ \Omega_{T,k}(\delta_T)  = \left\{ \left| \frac{  N^k[B_1, T-B_2] }{ T } - \mu_k^0 \right| \leq  \delta_T\right\}$$
 and recall that 
 $$ \Omega_T = \cap_{k\in [K]} \Omega_{T,k}(\delta_T.$$

\subsubsection{Lemma \ref{lem:devNT}}
\begin{lemma}\label{lem:devNT}
Under the Hypotheses \ref{cond:suprho}-\ref{cond:sparsity2}  and if $\log K = o(\sqrt{T})$, we have  for all $B_1, B_2$ fixed and
for any $o(1) = \delta_T >> (\log K)/\sqrt{T}$.
 $$ \mathbb P_0\left( \max_k \left| \frac{  N^k[B_1, T-B_2] }{ T } - \mu_k^0 \right| >\delta_T\right) \leq e^{-\frac{ \bar \gamma \delta_T^2 T}{ 8 (\log K) }} + 2T^{-1/2},$$ 
 where $\bar \gamma$ depends on $c, R_1, R_\infty, c_0, C_0 $.
\end{lemma}

\begin{proof}[Proof of Lemma \ref{lem:devNT}]
First we write 
We split $[B_1, T-B_2]$ in $J_T \leq   \frac{ Tt(c) }{ 3\log ( \|\nu^0\|_1  )} = J_0 T/\log K$  equal size intervals $J_0 = \frac{t(c)}{3} \frac{ \log K}{ \log (\|\nu^0\|_1)} =\frac{t(c)}{3}( 1 + o(1)) $. Each interval has size $(T-B_2-B_1)/J_T$ and for the sake of simplicity and without loss of generality we can set $B_1=B_2=0$ so that the intervals are $ [mT/J_T, (m+1)T/J_T]$, $m=0, \cdots , J_T-1$. We split 
$ [mT/J_T, (m+1)T/J_T] = [mtJ_T, (2m+1)T/(2J_T)] \cup [(2m+1)T/(2J_T),(m+1)T/J_T] $ and we write 
\begin{align*}
\frac{N^k[0,T]}{T}-\mu_k^0 &= \sum_{m=0}^{J_T-1}\frac{  N^k[mtJ_T, (2m+1)T/(2J_T)] - \mu_k^0 T/(2J_T) }{ T} \\
  & \quad + \sum_{m=0}^{J_T-1}\frac{  N^k[(2m+1)T/(2J_T),(m+1)T/J_T]  - \mu_k^0 T/(2J_T) }{ T}
  \end{align*}
  and we study each term separately using the same method. 
  Let $Z_{m,k}= N^k[(2m+1)T/(2J_T),(m+1)T/J_T]  - \mu_k^0 T/(2J_T)$ and construct 
  $\tilde Z_{m,k}= N^{0,k,m}[(2m+1)T/(2J_T),(m+1)T/J_T]  - \mathbb E_0(N^{0,k,m}[(2m+1)T/(2J_T),(m+1)T/J_T] )$
  where  $N^{0,k,m}$ is the subprocess of $N^k$ constructed in Lemma \ref{lem6:donnetsup}. Writing $\bar N^{k,m} = N^k - N^{0,k,m}$ and $B_m = [(2m+1)T/(2J_T),(m+1)T/J_T]$, for all $k$, we have 
  \begin{align*}
 \sum_{m=0}^{J_T-1} |Z_{m,k} - \tilde Z_{m,k} |& \leq \sum_{m=0}^{J_T-1}\bar N^{k,m}[(2m+1)T/(2J_T),(m+1)T/J_T] \\
  & \quad + \mathbb E_0(\bar N^{k,m}[(2m+1)T/(2J_T),(m+1)T/J_T] ) \\
 & \leq \sum_{m=0}^{J_T-1} [\bar N^{m}(I_m) + \mathbb E_0 [\bar N^{m}(I_m) ].
\end{align*}
Using Lemma \ref{lem6:donnetsup} we have choosing $x = T^{-1/2}$, 
$$\mathbb P_0 \left( \frac{\sum_{m=0}^{J_T-1} \bar N^{m}(I_m) }{T}> T^{-1/2} \right) \leq ee^{- t(c) T^{1/2}}$$
and we also have 
\begin{align*}
 \mathbb E_0 \left( \frac{\sum_{m=0}^{J_T-1} \bar N^{m}(I_m) }{T} \right)& \leq T^{-1/2} + \int_{T^{-1/2}}^\infty \mathbb P_0 \left( \frac{\sum_{m=0}^{J_T-1} \bar N^{m}(I_m) }{T}> x \right) dx \\
  & \leq T^{-1/2}+ \frac{ ee^{- t(c) T^{1/2} } }{ t(c) } \leq 2T^{-1/2}
  \end{align*}
  for $T$ large enough (depending only on $c$).
We thus obtain that 
$$\frac{ \sum_{m=0}^{J_T-1} Z_{m,k}}{ T}  = \frac{\sum_{m=0}^{J_T-1} \tilde Z_{m,k} }{ T} + O_{P_0}( T^{-1/2})$$
uniformly in $k$ and $\tilde Z_{m,k} = N^{0,k,m}[(2m+1)T/(2J_T),(m+1)T/J_T]  - \mathbb E_0(N^{0,k,m}[(2m+1)T/(2J_T),(m+1)T/J_T] )$ are iid random variables.

 We have 
 $$ \mathbb E_0(N^{0,k,m}[(2m+1)T/(2J_T),(m+1)T/J_T] )\leq \mathbb E_0(N^{k}[(2m+1)T/(2J_T),(m+1)T/J_T] ) \leq T \mu_k^0/2J_T.$$  We now study with $B =  T/(2J_T)$ and $t < t(c)$ defined in Lemma \ref{lem:lambda2}
 \begin{align*}
  I_k(m) & = \mathbb E_0( e^{tN^{0,k,m}[(2m+1)T/(2J_T),(m+1)T/J_T]}) \leq  \mathbb E_0( e^{tN^{k}[(0,B]}) \\
  & \leq  \exp\left( B\gamma t \right) = \exp\left( \frac{ T\gamma t}{ 2J_T} \right).
  \end{align*}
  This leads to 
    \begin{align*}
\mathbb P_0 \left( \frac{ | \sum_{m=1}^{J_T} \tilde Z_{m,k}| }{T} > \delta_T \right)  &\leq
\mathbb P_0 \left( \frac{  \sum_{m=1}^{J_T} \tilde Z_{m,k} }{T} > \delta_T \right) + \mathbb P_0 \left( \frac{  -\sum_{m=1}^{J_T} \tilde Z_{m,k} }{T} > \delta_T \right) \\
& \leq e^{- t \delta_T T} \mathbb E_0(e^{ t \tilde Z_{m,k}})^{J_T} +
e^{- t \delta_T T} \mathbb E_0(e^{ -t \tilde Z_{m,k}})^{J_T} .
\end{align*}
Let $t < t(c)$, then 
\begin{align*}
\mathbb E_0(e^{ t \tilde Z_{m,k}}) &\leq  e^{-t 1\mathbb E_0(N^{0,k,m}[(2m+1)T/(2J_T),(m+1)T/J_T] )} \times \\
& \left( 1 + t \mathbb E_0( N^{0,k,m}[(2m+1)T/(2J_T),(m+1)T/J_T] )+ t^2 \mathbb E_0(  N^{k}[0,B]^2  e^{tN^{k}[0,B]})  \right)\\
& \leq e^{ t^2 \mathbb E_0(  N^{k}[0,B]^2  e^{tN^{k}[0,B]})}. 
\end{align*}
We choose $2t < t_0 < t(c)$, then using $x^2 \leq e^{tx}/t_0^2$,   we bound
$$\mathbb E_0(  N^{k}[0,B]^2  e^{tN^{k}[0,B]} ) \leq \frac{ 1 }{ t_0^2 } \mathbb E_0(   e^{2t_0N^{k}[0,B]} ) \leq \frac{ e^{ T\gamma t_0 /J_T }}{t_0^2 }  $$
which in turns implies that  for $t_0= J_T/T=o(1)$,  and $\delta_T=o(1)$ 
  \begin{align*}
\mathbb P_0 \left( \frac{ \sum_{m=1}^{J_T} \tilde Z_{m,k} }{T} > \delta_T \right)
& \leq 
e^{- t \delta_T T+  t^2 \frac{T^2e^{\gamma}}{ J_T} } \leq e^{ - \frac{ J_T \delta_T^2 e^{-\gamma} }{4} }= o(1/K)
\end{align*}
as soon as  $ \log K = o( T^{1/2}) $  and $o(1) = \delta_T >> \sqrt{ (\log K)/J_T}$, since $J_T = J_0T/ \log K $ if $K>> 1$ or $J_T = J_0 T/\log T$ if $K$ is bounded. The same computation allow also  to bound 
  \begin{align*}
\mathbb P_0 \left( \frac{- \sum_{m=1}^{J_T} \tilde Z_{m,k} }{T} > \delta_T \right)
& \leq 
 e^{ - \frac{ J_T \delta_T^2  }{4\bar C_0} }= o(1/K)
\end{align*}
which terminates the proof of Lemma \ref{lem:devNT}.  
\end{proof}

\section*{Acknowledgments.}

We would like to thank Th\'eo Leblanc for useful discussions on the sharpness of the conditions to derive exponential moments for the number of events.

\bibliographystyle{siamplain}
\bibliography{biblio} 

\begin{thebibliography}{10}

\bibitem{bacry-etal-2020}
{\sc E.~Bacry, M.~Bompaire, S.~Ga\"iffas, and J.-F. Muzy}, {\em Sparse and
  low-rank multivariate {H}awkes processes}, J. Mach. Learn. Res., 21 (2020),
  pp.~Paper No. 50, 32.

\bibitem{cai2024latent}
{\sc B.~Cai, J.~Zhang, and Y.~Guan}, {\em Latent network structure learning
  from high-dimensional multivariate point processes}, Journal of the American
  Statistical Association, 119 (2024), pp.~95--108.

\bibitem{chen2019}
{\sc S.~Chen, A.~Shojaie, E.~Shea-Brown, and D.~Witten}, {\em The multivariate
  hawkes process in high dimensions: Beyond mutual excitation}, 2019,
  \url{https://arxiv.org/abs/1707.04928},
  \url{https://arxiv.org/abs/1707.04928}.

\bibitem{Chen2017}
{\sc S.~Chen, D.~Witten, and A.~Shojaie}, {\em {Nearly assumptionless screening
  for the mutually-exciting multivariate Hawkes process}}, Electronic Journal
  of Statistics, 11 (2017), pp.~1207--1234,
  \url{https://doi.org/10.1214/17-EJS1251}.

\bibitem{daley01}
{\sc D.~J. Daley}, {\em The busy period of the m/gi/? queue}, Queueing Systems,
   (2001).

\bibitem{donnet18}
{\sc S.~Donnet, V.~Rivoirard, and J.~Rousseau}, {\em {Nonparametric Bayesian
  estimation for multivariate Hawkes processes}}, The Annals of Statistics, 48
  (2020), pp.~2698 -- 2727, \url{https://doi.org/10.1214/19-AOS1903},
  \url{https://doi.org/10.1214/19-AOS1903}.

\bibitem{donnet18:supplement}
{\sc S.~Donnet, V.~Rivoirard, and J.~Rousseau}, {\em {Nonparametric Bayesian
  estimation for multivariate Hawkes processes}}, The Annals of Statistics, 48
  (2020), pp.~2698 -- 2727, \url{https://doi.org/10.1214/19-AOS1903},
  \url{https://doi.org/10.1214/19-AOS1903}.

\bibitem{ghosal00}
{\sc S.~Ghosal, J.~K. Ghosh, and A.~W. van~der Vaart}, {\em {Convergence rates
  of posterior distributions}}, The Annals of Statistics, 28 (2000), pp.~500 --
  531, \url{https://doi.org/10.1214/aos/1016218228},
  \url{https://doi.org/10.1214/aos/1016218228}.

\bibitem{ghosal:vdv:07}
{\sc S.~Ghosal and A.~van~der Vaart}, {\em Convergence rates of posterior
  distributions for non iid observations}, The Annals of Statistics, 35 (2007),
  pp.~192--223.

\bibitem{Hoffmann2013OnAP}
{\sc M.~Hoffmann, J.~Rousseau, and J.~Schmidt-Hieber}, {\em On adaptive
  posterior concentration rates}, The Annals of Statistics, 43 (2015),
  pp.~2259--2295.

\bibitem{kreiss2025commondriverssparselyinteracting}
{\sc A.~Kreiss, E.~Mammen, and W.~Polonik}, {\em Common drivers in sparsely
  interacting hawkes processes}, 2025, \url{https://arxiv.org/abs/2504.03916},
  \url{https://arxiv.org/abs/2504.03916}.

\bibitem{leblanc2025}
{\sc T.~Leblanc}, {\em Sharp and exact tail estimates for multitype poissonian
  galton watson processes and inhomogeneous cluster processes}, 2025,
  \url{https://arxiv.org/abs/2507.08462},
  \url{https://arxiv.org/abs/2507.08462}.

\bibitem{Rousseau2016AsymptoticF}
{\sc J.~Rousseau and B.~Szabo}, {\em {Asymptotic frequentist coverage
  properties of Bayesian credible sets for sieve priors}}, The Annals of
  Statistics, 48 (2020), pp.~2155 -- 2179,
  \url{https://doi.org/10.1214/19-AOS1881},
  \url{https://doi.org/10.1214/19-AOS1881}.

\bibitem{sulem2024bayesian}
{\sc D.~Sulem, V.~Rivoirard, and J.~Rousseau}, {\em Bayesian estimation of
  nonlinear hawkes processes}, Bernoulli, 30 (2024), pp.~1257--1286.

\bibitem{Tang-Li-2023}
{\sc X.~Tang and L.~Li}, {\em Multivariate temporal point process regression},
  J. Amer. Statist. Assoc., 118 (2023), pp.~830--845,
  \url{https://doi.org/10.1080/01621459.2021.1955690},
  \url{https://doi.org/10.1080/01621459.2021.1955690}.

\bibitem{wang2025statistical}
{\sc X.~Wang, M.~Kolar, and A.~Shojaie}, {\em Statistical inference for
  networks of high-dimensional point processes}, Journal of the American
  Statistical Association, 120 (2025), pp.~1014--1024.

\end{thebibliography}

\end{document}